\documentclass[onefignum,onetabnum]{siamart171218}
\usepackage{amsmath, latexsym}
\usepackage{amssymb}
\usepackage[utf8]{inputenc}
\usepackage[english]{babel}
\usepackage{tikz, graphicx}
\usetikzlibrary{shapes.geometric, arrows}
\usetikzlibrary{positioning}
\usepackage{textpos}
\usepackage{amsfonts}
\usepackage{newlfont}
\usepackage{enumerate}
\usepackage{epsfig}
\usepackage{subfig}
\usepackage{subcaption}
\usepackage{float} 
\usepackage{multirow}
\usepackage{ntheorem}
\usepackage{mathrsfs}

\newsiamremark{remark}{Remark}

\def\S{\mathcal{S}}
\def\T{T\wedge\tau_N^n}
\def\L{\mathcal{L}}

\def\I{\mathcal{I}}
\def\F{\mathcal{F}}

\def\B{\mathcal{B}}
\def\D{\mathcal{D}}

\def\Y{\mathrm{Y}}
\def\E{\mathcal{E}}
\def\P{\mathcal{P}}
\def\G{\mathcal{G}}
\def\T{\mathcal{T}}
\def\Lp{L^{p'}(I';L^p(I))}
\def\X{\mathbb{X}}

\def\V{\mathfrak{B}}

\def\U{\mathcal{U}}
\def\X{\mathcal{X}}
\def\Y{\mathcal{Y}}
\def\Z{\mathcal{Z}}
\def\I{\mathcal{I}}
\def\OMT{\Omega_T}

\def\BV{ B^\theta_{p'}(0,T;B^s_p(\Omega))}
\def\Sf{S_{k,k'}(f)}
\def\Ssf{S^*_{k,k'}(f)}

\def\M{\mathcal{M}}
\def\N{\mathcal{N}}

\def\arraystretch{1.15}

\ifpdf
\hypersetup{
	pdftitle={An Example Article},
	pdfauthor={h}
}
\fi

\author{Shiv Mishra \thanks{``Department of Mathematics, Indian Institute of Technology Roorkee, Roorkee 247667, India."\\
		E-mail: {shiv\_m@ma.iitr.ac.in}, {arbaz@ma.iitr.ac.in}. \\
		\textbf{Funding:} ``Shiv Mishra is supported by UGC, India."}
	\and
	Arbaz Khan \footnotemark[1]}
\ifpdf
\hypersetup{pdftitle={New \LaTeX\ Style} ,pdfauthor={S. Mishra and A. Khan}}
\fi

\begin{document}
	\title{A priori error analysis of consistent PINNs for parabolic PDEs}
\date{Date: \today}
\maketitle
\begin{abstract}
	We present a new a priori analysis of a class of collocation methods for parabolic PDEs that rely only on pointwise data of force term \( f \), boundary data \( g \), and initial data \( u_0 \). Under Besov regularity assumptions, we characterize the optimal recovery rate of the solution \( u \) based on sample complexity. We establish error bounds by constructing a new consistent loss function \( \L^* \) that effectively controls the approximation error. This loss incorporates contributions from the interior, boundary, and initial data in a discretized form and is designed to reflect the true PDE structure. Our theoretical results show that minimizing \( \L^* \) yields near-optimal recovery under suitable regularity and sampling. Novel practical variants of \( \L^* \) are discussed, and numerical experiments confirm the framework’s effectiveness.
\end{abstract}



\section{Introduction}	
The parabolic partial differential equations (PDEs), such as the heat equation and its variants are fundamental in modelling diffusion-dominated phenomena in physics, biology, and engineering. Reliable numerical methods for such problems require not only accurate approximations but also rigorous control over approximation errors. A well-established theory of convergence results is already available for traditional numerical schemes, such as finite difference or finite element methods but  physics-informed neural networks (PINNs) are still under development, have very few theoretical findings and have several challenges, including optimization difficulties, inconsistent residual formulations, and generalization issues. Recently, by using the given physical laws in the loss function of a neural network, PINNs have gained attention as a mesh-free, data-driven approach for solving PDEs and developing a flexible framework to integrate both observed data and model constraints. Their ability to handle complex geometries, sparse data, and high-dimensional problems have made them attractive across various domains of science and engineering \cite{ cuomo2022scientific, karniadakis2021physics}. 

In this paper, we discuss about a parabolic equation with nonhomogeneous Dirichlet boundary condition over the domain $\Omega_T:= \Omega\times (0,T]$, where $\Omega\subset \mathbb{R}^d$ is a bounded domain with a Lipschitz boundary $\partial\Omega$  and $\Sigma := \partial\Omega\times (0,T]$ is its lateral boundary. The model is given by
\begin{align}\label{1.1}
	\begin{cases}
		u_t(x,t)- \Delta u(x,t) &= f(x,t) \quad \text{in } \Omega_T,\\ 
		\hspace{1cm}u(x,t) &= g(x,t) \quad \text{on } \Sigma,\\ 
		\hspace{1cm}u(x,0) &= u_0(x) \quad \text{in} \, \Omega\times \{t=0\},
	\end{cases}
\end{align}
where $f$ is the given force term, $g$ is boundary data and $u_0$ is the initial value of $u$ at $t=0.$ This classical parabolic initial-boundary value problem models time-evolving diffusion processes with prescribed boundary and initial conditions. The primary goal is to accurately approximate the solution \( u \) from given data \( f \), \( g \), and \( u_0 \), which are typically only available at discrete locations in practical applications.

\subsection{Literature survey}
The numerical schemes for solving parabolic PDEs depend on classical methods such as finite difference, finite element, and spectral techniques~\cite{thomee2006galerkin} for a long time. These methods have available stability and convergence theory but are computationally expensive for high-dimensional domains or with complex geometries. 
There are some neural network-based methods like the deep Ritz method~\cite{yu2018deep} and the finite neuron method~\cite{xu2020finite} have been introduced,  which overcome these challenges and aim to approximate PDE solutions using neural architectures with variational principles~\cite{lu2022priori, lu2021priori, muller2022error, siegel2023greedy}. 

Among them, PINNs introduced by Raissi et al.~\cite{karniadakis}, have attracted attention for solving PDEs and showing strong empirical performance~\cite{cai2021physics, cuomo2022scientific, karniadakis2021physics}. However, the theoretical understanding of PINNs remains limited. Recently, there has been some work towards this objective, as shown in references \cite{gazoulis2023stability, shin2023error, zeinhofer2024unified}. Shin et al.~\cite{SYDK} gave convergence results for elliptic PDEs, Mishra et al.~\cite{MSMR} studied generalization using statistical learning tools, and Lu et al.~\cite{lu} proposed a residual minimization framework that includes convergence for parabolic PDEs under Sobolev regularity. Additionally, Akram et al.~\cite{akram2025errorestimatesviscousburgers} derived error bounds for the viscous Burgers equation under various Sobolev norms. Although for analyses exact boundary condition enforcement or require specific architectures~\cite{berg2018unified, jagtap2020conservative} is assumed, and they typically do not provide convergence rates in terms of network parameters or collocation points, especially for deterministic sampling. To develop  the best possible approximations under regularity assumptions~\cite{dahlkeoptimal, krieg2022recovery, novak2006function}, an alternative way is proposed with optimal recovery theory. DeVore et al.~\cite{devore1989optimal, devore1993constructive} established recovery rates under Besov regularity for elliptic PDEs, and proposed regularity framework that has also been extended to parabolic problems in the works of Lunardi and Amann~\cite{MR1242579, amann1997operator, amann2009anisotropic, lunardi2012analytic}.

Recent developments include convergence analysis for elliptic PDEs in terms of both the number of collocation points and the size of the neural network under general Besov smoothness assumptions~\cite{BVPS}. However, there is currently no comprehensive study providing similar error estimates for PINNs applied to parabolic PDEs in this setting. This motivates our work, where we develop a rigorous a priori analysis of consistent PINNs using optimal recovery theory.

\subsection{Principal contributions}
This paper introduces a novel framework of consistent PINNs (CPINNs) for parabolic PDEs, specifically for the heat equation with nonhomogeneous Dirichlet boundary conditions. Unlike existing PINN approaches, under the CPINN framework, we construct a loss function $\L^*$ consistent with the given PDE whose minimization over expressive neural networks achieves near-optimal recovery. To the best of our knowledge, there is no a priori analysis exists for time-dependent PDEs in terms of collocation points under Besov regularity. We are the first ones to provide a rigorous error analysis and establish theoretical recovery guarantees under the Besov model class assumptions in space as well as time for the heat equation. The novelty of this paper lies in characterizing the approximation performance of a function in the space-time domain using Besov model classes and deriving the lower bounds and convergence rates in both interior and boundary norms. Furthermore, we define an error bound for the solution from the proposed consistent loss.
\subsection{Outline of paper}
The rest of the paper is structured as follows: Section \ref{Sec_2} gives Sobolev and Besov spaces with the numerical framework of collocation methods. The approximation properties, followed by the optimal recovery framework with theoretical proofs are given in sections \ref{sec_2} and \ref{sec_3}. In section \ref{Sec_5} we introduce the consistent loss function along with the discrete norms and establish error control and convergence results. Section \ref{Sec_6} provides supporting numerical examples, and section \ref{Sec_7} concludes with a summary and future directions. Additional technical details are provided in the appendix.
\section{Function spaces and numerical framework}\label{Sec_2}
\subsection{Function spaces}
In this section, we first introduce the function spaces required for the analysis. $W^{k,p}(\Omega)$ denotes the standard Sobolev spaces for scalar valued functions with norms $\|\cdot\|_{W^{k,p}(\Omega)}$ for $k \ge 0$ and $1 \le p \le \infty$. For $k = 0$, we write $W^{0,p}(\Omega)=L^p(\Omega)$ with norm $\|\cdot\|_{p}$, and for $p=2$, we write $W^{k,2}(\Omega)=H^{k}(\Omega)$. Furthermore, we define $H^1_0(\Omega) := \{v \in H^1(\Omega) \ |\ v|_{\partial\Omega} = 0\}$, $H^{-1}(\Omega)$ as the dual space of $H^1_0(\Omega)$ and $H^{1/2}(\partial\Omega)$ is defined as the trace of functions in $H^1(\Omega)$.
Now, we introduce the following spaces:
\begin{itemize}
	\item $H^{1/2}(\Omega) := \left\{ v \in L^2(\Omega) \ \middle| \ \int_{\Omega} \int_{\Omega} \frac{|v(x) - v(y)|^2}{|x - y|^{d + 1}} \, dx \, dy < \infty \right\}$.
	\item $L^2(0,T; H^1(\Omega)) := \left\{ u : [0,T] \to H^1(\Omega) \ \middle| \ \int_0^T \|u(t)\|_{H^1(\Omega)}^2 \, dt < \infty \right\}$.
	\item $H^{\frac{1}{2},\frac{1}{4}}(\Sigma) := L^2(0,T;H^{\frac{1}{2}}(\partial\Omega))\cap H^{\frac{1}{4}}(0,T;L^2(\partial\Omega))$.
\end{itemize}
\textbf{Besov space:} 
Let \( 0 < \theta < \infty \), \( 1 \leq p',q' < \infty \), and  $E$ be a Banach space over $\Omega$. Then we introduce Besov spaces  \cite[section~7.2]{AH} on \( J =[0,T] \) as 
\begin{align*}
	\B^{\theta}_{p'q'}(J;E) = \left\{\hspace{-0.5mm} f\hspace{-0.5mm}\in W^{[\theta]}_{q'}(J;E)\hspace{-0.5mm} \mid  |f|_{B^{\theta}_{p'q'}} \hspace{-1mm}= \hspace{-1mm} \Bigg( \hspace{-1mm}\int_0^{T} \hspace{-1.5mm}\bigg( \int_0^{T} \hspace{-0.5mm}\frac{|\Delta^{r'}_{h'} f(t)|_E^{p'}}{h'^{\theta p'}} \, dt \bigg)^{\frac{q'}{p'}} \frac{dh'}{h'} \Bigg)^{\frac{1}{q'}} \hspace{-1mm}< \infty \right\} ,
\end{align*}
where $\Delta^{r'}_{h'}$ denotes forward difference of an integer order ${r'>\theta}$ and step length $h'$ in time. The corresponding norm on space is defined as 
\begin{align}
	\|f\|_{B^\theta_{p'q'}(J;E)} = \|f\|_{W^{[\theta]}_{p'}}+ |f|_{B^{\theta}_{p'q'}}.
\end{align}
For \( q' = \infty \), the integral with respect to \({dh'}/{h'} \) is to be replaced by \( \sup_{0 < h' < T} \) and space is denoted as $B^{\theta}_{p'}(J;E)$.
From the definition of Besov space, we get
\begin{align}
	|f|^{p'}_{B^\theta_{p'}(0,T;B^s_p(\Omega))}
	\asymp \sup_{k'\geq 0}\sup_{k\geq 0} \frac{\omega_{r,r'}(f,2^{-k}\times 2^{-k'})^{p'}_{L^{p'}(0,T;L^p(\Omega))}}{2^{-k'\theta p'}}\ 2^{ksp'},
\end{align}
where \(\omega_{r,r'}(f,b\times b')_{L^{p'}(0,T;L^p(\Omega))}:= \sup_{|h'|\leq b'}\sup_{|h|\leq b}\|\Delta^{r}_h\Delta^{r'}_{h'} f(x,t)\|_{L^{p'}(0,T;L^p(\Omega))}\),  \(b,b'>0\), defined as the modulus of smoothness of $f$. It follows that a function $f\in B^\theta_{p'}(0,T;B^s_p(\Omega))$ if and only if 
\begin{align}
	\omega_{r,r'}(f,2^{-k}\times 2^{-k'})_{L^{p'}(0,T;L^p(\Omega))} \leq 2^{-(ks+k'\theta)p'} 	|f|^{p'}_{B^\theta_{p'}(0,T;B^s_p(\Omega))}.
\end{align}

To ensure the existence and uniqueness of a solution to \eqref{1.1}, we need to apply assumptions on $f, g,$ and $u_0$. Let us suppose  $f \in L^2(0,T;H^{-1}(\Omega)), g \in H^{1/2,1/4}(\Sigma)$ and $u_0 \in L^2(\Omega)$. As discussed in \cite[Section~15.5]{lions1972}, under these assumptions there exists a unique solution $u \in V$ of \eqref{1.1}, where
$$V := \{v\mid v\in L^2(0,T;H^1(\Omega)),\, v_t \in L^2(0,T;H^{-1}(\Omega))\},$$ and $u$ satisfies the estimate
\begin{align}\label{1.3}
	c &\left( \| f \|_{L^2(0,T;H^{-1}(\Omega))} + \| g \|_{ H^{1/2,1/4}(\Sigma)} + \|u_0\|_{L^2(\Omega)}\right) 
	\leq \| u \|_{L^2(0,T;H^1(\Omega))} \\ \notag &+ \|u_t\|_{L^2(0,T;H^{-1}(\Omega))} 
	\leq C  \left( \| f \|_{L^2(0,T;H^{-1}(\Omega))} + \| g \|_{ H^{1/2,1/4}(\Sigma)} + \|u_0\|_{L^2(\Omega)}\right),
\end{align}
where the constants $c, C$ depend on $\Omega$. Thus, for the theoretical loss function
\begin{align}
	\mathcal{L}_T(v) := \| f + \Delta v - v_t\|_{L^2(0,T;H^{-1}(\Omega))} + \| g - v \|_{H^{1/2,1/4}(\Sigma)} + \|u-u_0\|_{L^2(\Omega)},
\end{align}
minimizing this loss over the whole of $V$ has $u$ as its unique solution.
\subsection{Numerical framework of collocation methods and PINNs}
Neural networks are increasingly used as nonlinear approximation spaces for solving PDEs such as \eqref{1.1}. Given a fixed-architecture class \( \N_n \) with \( n \) parameters, an approximate solution \( \hat{u} \in \N_n \) is computed by minimizing a norm \( \| \cdot \|_X \) based on sampled input data \( f \), \( g \), and \( u_0 \) over the space-time domain \( \Omega_T \) as
\begin{align*}
	\begin{cases}
		f = (f_{11},f_{12}, \dots, f_{1\hat m}, f_{21},\dots f_{\tilde m\hat m}), \ f_{i,j} := f(x_i,t_j), \ i = 1, \dots, \tilde m,\ j = 1, \dots \hat{m};\\
		g = (g_{11},g_{12} \dots, g_{1\hat{m}},g_{21},\dots g_{\bar{m}\hat{m}}), \ \hspace{1.6mm}g_{i,j} := g(x_i,t_j), \ i = 1, \dots, \bar{m},\  j = 1, \dots \tilde{m};\\
		u_0= (u_{10}, \dots, u_{\tilde m0}), \quad \hspace{2.3cm}u_i := u(x_i,0), \hspace{2mm}  i= 1, \dots,\tilde  m;  
	\end{cases}
\end{align*}
where \( (x_i, t_j) \in \Omega_T \) represent the sample locations. These points are referred to as data sites and are denoted by
\begin{align*}
	\begin{cases}
		\mathcal{X} &:= \{(x_1,t_1),(x_1,t_2),\dots ,(x_1,t_{\hat m}), \dots (x_{\tilde m}, t_{\hat m})\},\\
		\mathcal{Y} &:= \{(x_1,t_1),(x_1,t_2),\dots ,(x_1,t_{\hat m}), \dots (x_{\bar m}, t_{\hat m})\},\\
		\mathcal{Z} &:= \{(x_1,t_0),(x_2,t_0),\dots ,(x_{\tilde m},t_0)\}.
	\end{cases}
\end{align*}  

PINNs is the numerical procedure to find a $\hat{u}\in \N_n$ which ‘fits the data’. A commonly adopted approach is to define $\hat{u}$ as one of the elements of the set
\begin{align}\label{1.5}
	u \in\mathop{\arg\min}\limits_{S \in \N_n} \mathcal{L}(S), \quad \\\notag
	\text{where,}\ \mathcal{L}(S) := \frac{1}{N_{int}} \sum_{i=1}^{N_{int}} \left[ \Delta S(y_i) + f(y_i) -S_t(y_i)\right]^2 
	&+ \frac{\lambda_1}{N_{sb}} \sum_{i=1}^{N_{sb}} \left[ S(y_i) - g(y_i) \right]^2+\frac{\lambda_2}{N_{tb}} \sum_{i=1}^{N_{tb}} \left[ S(y_i) - u_0(y_i) \right]^2,
\end{align}
with $\lambda_1,\lambda_2$ being tuning parameters, $y_i=(x,t)_i$ be data sites, $N_{int} = \tilde m\times \hat m, N_{sb} = \bar m \times \hat m$, and $N_{tb} = \tilde m$. To simplify the analysis, we will assume $\lambda_1,\lambda_2=1$  throughout the paper. 
We identify a more suitable loss function, \( \mathcal{L}^* : V \to \mathbb{R} \), for use in collocation based methods. The squared value of the loss function is given as 
\begin{align*}
	\mathcal{L}_{sq}^*(v) &:=\frac{1}{\hat m}\sum_{j = 1}^{\hat m}\hspace{-1mm}\left[\frac{1}{\tilde m}\sum_{i = 1}^{\tilde m}|f(x_i,t_j) + \Delta v(x_i,t_j) -v_t(x_i,t_j) |^\gamma \right] ^{\frac{2}{\gamma}}\hspace{-2mm}+ \hspace{-1mm}\frac{1}{\tilde m} \sum_{j=1}^{\tilde{m}} |v(x_j,t_0) - u_0|^2  \\
	& \hspace{-1cm}+   \frac{2}{\bar m\hat{m}} \sum_{j= 1}^{\hat{m}}\sum_{i=1}^{\bar m} \left| g(x_i,t_j)  - v(x_i,t_j)  \right|^2
	+ \frac{1}{\hat{m}\bar{m}^2}\sum_{l=1}^{\hat{m}}\sum_{i\neq j}^{\bar{m}}\frac{ | [g - v](x_i,t_l)- [g - v](x_j,t_l)|^2}{|(x_i,t_l)-(x_j,t_l)|^d} \\
	& \hspace{4.5cm}+  \frac{1}{\hat{m}^2\bar{m}}\sum_{j\neq l}^{\hat{m}}\sum_{i}^{\bar{m}}\frac{ | [g - v](x_i,t_l)- [g - v](x_i,t_l)|^2}{|(x_i,t_j)-(x_i,t_l)|^{3/2}},
\end{align*}
where  \( \gamma \) denotes the smallest number for which the embedding \( L^\gamma(\Omega) \hookrightarrow H^{-1}(\Omega) \) holds. The loss function \( \mathcal{L}^* \) is designed to be consistent with recovering \( u \) in the target space \( V \), making it the appropriate choice for minimization, the resulting method is referred to as consistent PINNs (CPINNs).

\section{Domain decomposition and polynomial interpolation}\label{sec_2}
Functions in Besov spaces can be characterized via piecewise polynomial approximation, which can be efficiently constructed using interpolation. For the domain $\Omega_T = (0,1)^d \times (0,T]$, we define a tensor-product grid $G_{r,r'} := \left\{\left(\frac{j_1}{r-1}, \cdots ,\frac{j_d}{r-1}, \frac{j_{d+1}}{r'-1}\right)\right\}$ over $\bar{\Omega}_T = [0,1]^d\times [0,T],$ where $ j_i\in \{0,1, \cdots ,r-1\}, \ i = 1, \cdots, d, \ j_{d+1} \in \{0,1, \cdots ,r'-1\},$ for $r,r'\in \mathbb{N},\ r,r'>1$ and apply a simplicial (Kuhn-Tucker) decomposition of $\bar{\Omega}_T$ into simplices $E \times I'$, enabling interpolation based approximation using only discrete data $(f_i), (g_i), (u_{0i})$. Let $\T_{k,k'}:= \T_k\times \I_{k'}$ be simplicial decomposition of $\bar{\Omega}_T$.

For any simplex \(E \times I'\), the number of grid points equals the dimension of the polynomial space defined as $\P_{r,r'}^{d+1} := \left\{\sum_{k'<r'}^{}\sum_{|k|_1<r}^{}a_{k,k'}x^kt^{k'}, \ a_{k,k'} \in \mathbb{R},\right\},$ where $x^k:= x_1^{k_1}\cdots x_d^{k_d}, k := (k_1, \dots, k_d), \ k_j\geq 0, \ |k|_1 = \sum_{j = 1}^{d}k_j,$  and  $k'\geq 0.$ 

As discussed in \cite{sudirham2006space}, polynomial interpolation using elements of \(\mathcal{P}_{r,r'}\) at the grid points in a reference simplex yields a bounded projection. Let $\D_k$ and $\I_{k'}$ be decomposition of spatial and time domain in dyadic cubes. For simplex \(E \times I'\in \T_{k,k'}\), obtained via affine transformation from the reference simplex, the Lagrange interpolant
\begin{align*}
	L_{E,I'}(f) \hspace{-0.7mm}:= \hspace{-1mm}\sum_{i,j} f(x_i, t_j)\, \phi_i(x)\, \phi_j(t),\quad \|L_{E,I'}(f)\|_{C(\bar{T} \times \bar{I}')} \hspace{-0.5mm} \leq  \hspace{-0.5mm}\Upsilon_r(E)\Upsilon_{r'}(I') \|f\|_{C(\bar{E} \times \bar{I}')},
\end{align*}
where $ \Upsilon_r:= \|\sum_{i = 1}^{n_r}|\phi_{i,E}(x)|\|_{C(\bar{E})}$, $\Upsilon_{r'}:= \|\sum_{j = 1}^{n_{r'}} |\phi_{j,I'}(t)| \|_{C(\bar{E}, I')}$ are the corresponding Lebesgue constants and $(x_i, t_j) \in G_{r,r'} \cap (E \times I')$. Moreover, for any polynomial \(P \in \mathcal{P}_{r,r'}\), the interpolant satisfies the approximation estimate
\[
\|f - L_{E,I'}(f)\|_{C(\bar{E} \times \bar{I}')} \leq (1 + \Upsilon_r(E)\Upsilon_{r'}(I')) \|f - P\|_{C(\bar{E} \times \bar{I}')}.
\]

Let $S^*_{k,k'}(f)$ denote the piecewise polynomial interpolant defined by
\begin{align}\label{2.16}
	S^*_{k,k'}(f):= \sum_{E\times I'\in\T_{k,k'}}^{}L_{E,I'}(f)\chi_E\chi_{I'}:=\sum_{I'\in\I_{k'}}^{}\sum_{I\in\D_{k}}^{}L_{I,I'}(f)\chi_I\chi_{I'},
\end{align}
where \( \chi_E,\,  \chi_{I'}\) are the associated characteristic functions. The following theorem estimates the accuracy of \(S^*_{k,k'}\).
\begin{theorem},\label{$thm_2.2$}
	Let \(f\in \mathfrak{B} := B^{\theta}_{p'}(0,T;B^s_p(\Omega))\)  with \(s>d/p,\, \theta>1/p'\)  and \(S^*_{k,k'}\) be the interpolation operator. Then
	\begin{enumerate}[(i)]
		\item The approximation error in $L^{\tau'}(0,T;L^\tau(\Omega))$ norm for \(1\leq p \leq \tau,\ 1\leq p' \leq \tau'\),
		\begin{align}\label{2.17}
			\|f-S^*_{k,k'}(f)\|_{L^{\tau'}(0,T;L^\tau(\Omega))}\leq C|f|_{ \mathfrak{B} }\,2^{-(k(s-\frac{d}{p}+ \frac{d}{\tau}) +k'(\theta -\frac{1}{p'}+\frac{1}{\tau'}))},
		\end{align}
		\item For \(1 \leq p, p' \leq 2\), the approximation error in \(L^2(0,T;H^1(\Omega))\) norm satisfies,
		\begin{align}\label{2.18}
			\|f-S^*_{k,k'}(f)\|_{L^2(0,T;H^1(\Omega))}\leq C|f|_{ \mathfrak{B} }\,2^{-\big(k(s-1-\frac{d}{p}+ \frac{d}{2})+k'(\theta -\frac{1} {p'}+\frac{1}{2})\big)},
		\end{align}
		\item For \(1\leq p,p'\leq \infty,\)  there is a continuous representative \(\tilde f \in  \text{Lip} (s-d/p, \theta-1/p')\) such that \(f=\tilde f\ a.e.\). i.e.,
		\begin{align}\label{2.19}
			\omega_{r,r'}(\tilde f, b\times b')_{C(\OMT)}\leq C|f|_{ \mathfrak{B}}\,b^{(s-\frac{d}{p})}b'^{(\theta-\frac{1}{p'})}, \quad b,b'>0,
		\end{align}
	\end{enumerate}
	where $C$ be a constant independent from \(f, k\) and $k'$.
\end{theorem}
\begin{proof}
	\textbf{Case (i):}
	By Theorem \ref{thm_12.2}, we need to prove only
	\begin{align}\label{12.26}
		\|\Sf - \Ssf\|_{L^{\tau'} (0,T;L^\tau(\Omega))} \hspace{-1mm}\leq  \hspace{-0.9mm}C \|f\|_{ \mathfrak{B} } 2^{-k(s - \frac{d}{p} + \frac{d}{\tau})- k'(\theta-\frac{1}{p}+\frac{1}{\tau'})},
	\end{align}
	where \( C \) depends on \( r, r', d \) but not on \( f, k, k' \). For each \( I \times I' \in \D_k \times \I_{k'} \), let \( P_{I,I'} \in \P_{r,r'} \) be the best \( L^p \) approximation to \( f \) and define \( \Sf = \sum_{I', I} P_{I,I'} \chi_I \chi_{I'} \). Then for \( E \subset I \),
	\begin{align*}
		\|\Ssf - \Sf\|_{L^\infty(E\times I')} \leq \|L_{E,I'} (f - P_{I,I'})\|_{L^\infty(E\times I')}\leq C \|f - P_{I,I'}\|_{L^\infty(I\times I')}.
	\end{align*}
	Thus, by Theorem \ref{thm_12.2}, we get
	\begin{align}\label{12.28}
		\|\Ssf - \Sf\|_{L^\infty(I\times I')}&\leq C \|f-P_{I,I'}\|_{L^\infty(I\times I')}\notag\\&\leq C \|f\|_{ \mathfrak{B} } 2^{-\left(k(s - \frac{d}{p})+k' (\theta- \frac{1}{p'})\right)},
	\end{align}
	hence 
	\[
		\|\Ssf - \Sf\|_{L^\infty(\OMT)} \leq C \|f\|_{ \mathfrak{B} } 2^{-\left(k(s - \frac{d}{p})+k' (\theta- \frac{1}{p'})\right)},
	\]
		proving \eqref{12.26} for \( \tau = \infty \). To extend to \( p \leq \tau < \infty \), \( p' \leq \tau' < \infty \), fix \( (x,t) \in I \times I' \in \D_k \times \I_{k'} \) and let \( J_j \times J_{j'} \) denote dyadic cubes containing \( (x,t) \). Define \( Q_{j,j'} := P_{J_{j+1},J_{j'+1}} - P_{J_j,J_{j'}} \). Then, as in \eqref{12.21},
	\begin{align}\label{12.30}
		|f(x,t) - P_{I,I'} (x,t)| &\leq \sum_{j' \geq k'}\sum_{j \geq k} \|Q_{j.j'}\|_{L^\infty(J_{j'+1}\times J_{j+1})}\notag \\ 
		&\leq C  \sum_{j' \geq k'} \sum_{j \geq k} 2^{\frac{jd}{p}+\frac{j'}{p}} \|f - P_{J_j,J_{j'}}\|_{L^{p'}(J_{j'};L^p(J_j))}.
	\end{align}
	Using the modulus of smoothness, this yields
	\begin{align}\label{12.31}
		\|f-P_{I,I'}\|_{L^\infty(I\times I')}\leq C\sum_{j' \geq k'}\sum_{j \geq k}2^{\frac{jd}{p}+\frac{j'}{p'}}\tilde \omega_{r,r'}(f,2^{-j}\times 2^{-j'})_{\Lp}.
	\end{align}
	Then, for \( \beta \in (0, s - \frac{d}{p}) \), \( \beta' \in (0, \theta - \frac{1}{p'}) \), and from \eqref{12.28} and \eqref{12.31},
	\begin{align}\label{12.35}
		\|\Sf -& \Ssf\|_{L^\infty(I\times I')}\leq C\|f - P_{I,I'}\|_{L^\infty(I\times I')} \leq C  
		\Bigg\{  \sum_{j' \geq k'}\sum_{j \geq k} 2^{-(j\beta+j'\beta ') q'} \Bigg\}^{ \frac{1}{q'}}  \notag\\
		&\Bigg\{  \sum_{j' \geq k'}\sum_{j \geq k} \left[ 2^{\left(j(\frac{d}{p'}+\beta)+j'(\frac{d}{p'}+\beta ')\right)} \tilde w_{r,r'} (f, 2^{-j}\times  2^{-j})^{p'}_{\Lp} \right]^{p'} \Bigg\}^{ \frac{1}{p'}}.
	\end{align}
	Now summing over all \( I \times I' \in \D_k \times \I_{k'} \), and using subadditivity of \( \tilde\omega \), we get
	\begin{align*}
		\|\Sf - \Ssf\|^{p'}_{L^{p'}(0,T';L^p(\Omega))} &\leq C\zeta\sum_{I' \in \I_{k'}}\sum_{I \in \D_k}\|\Sf - \Ssf\|^{p'}_{L^\infty(I\times I'))} \\
		&\leq C\ \zeta \sum_{j \geq k} \sum_{j \geq k} 2^{(jd+j')} w_{r,r'} (f, 2^{-j}\times 2^{-j'})^{p'}_{L^p (0,T;L^p(\Omega))} \\ 
		&\leq C |f|^{p'}_{\mathfrak{B} } \zeta \sum_{j' \geq k'}\sum_{j\geq k} 2^{(jd+j')} 2^{-(js+j'\theta)p'} 
		\leq C |f|^{p'}_{\mathfrak{B} } 2^{-(ks+k'\theta)p'},
	\end{align*}
	where \( \zeta = 2^{-(kd + k')} \) . This proves \eqref{12.26} for \( \tau = p, \tau' = p' \). For general \( \tau \in (p,\infty), \tau' \in (p',\infty) \)
	\begin{align}\label{12.33}
		|\Sf - \Ssf|^{\tau'} &\leq |\Sf - \Ssf|^{\tau' - p'} |\Sf - \Ssf|^{p'}\notag \\ 
		&\hspace{-0.1cm}\leq C |f|^{\tau' - p'}_{\mathfrak{B}} 2^{-[k(s - \frac{d}{p})+k'(\theta-\frac{1}{p'})](\tau' - p')} |\Sf - \Ssf|^{p'}.
	\end{align}
	Integrating this expression completes the proof
	\begin{align}\label{12.34}
		\|\Sf - \Ssf\|^{\tau'}_{L^{\tau'} (0,T;L^\tau(\Omega))} \leq C |f|^{\tau'}_{\mathfrak{B}} 2^{-[k(s - \frac{d}{p} + \frac{d}{\tau})+ k'(\theta- \frac{1} {p'} + \frac{1}{\tau'})] \tau'}.
	\end{align}
	
	\textbf{Case (ii):}
	Since \( B^\theta_{p'q'}(0,T;B^s_{pq}(\Omega)) \subset B^\theta_{p'}(0,T;B^s_p(\Omega)) := B^\theta_{p'\infty}(0,T;B^s_{p\infty}(\Omega)) \), it is enough to prove the result for \( f \in \BV \), i.e., when \( q = q' = \infty \). From case (i), we know for $k,k' \geq 0$
	\begin{align}
		\|f - \Ssf\|_{L^2(0,T;L^2(\Omega))} \leq C |f|_{\mathfrak{B}}2^{-k(s - \frac{d}{p} + \frac{d}{2}) - k'(\theta - \frac{1}{p'} + \frac{1}{2})}.\label{12.37}
	\end{align}
	Define
	\begin{align}\label{12.38}
		R^*_{k,k'}(f) := \Ssf - S^*_{k-1,k'}(f) - S^*_{k,k'-1}(f) + S^*_{k-1,k'-1}(f), \quad k,k'\geq 1.
	\end{align}
	Then, by \eqref{12.37},
	\begin{align}\label{12.39}
		\|R^*_{k,k'}(f)\|_{L^2(0,T;L^2(\Omega))} \leq C |f|_{\mathfrak{B}} 2^{-k(s - \frac{d}{p} + \frac{d}{2}) - k'(\theta - \frac{1}{p'} + \frac{1}{2})}.
	\end{align}
	Since \( R^*_{k,k'}(f) \) is piecewise polynomial on each \( T \times I' \in \T_{k,k'} \), inverse inequalities yield
	\begin{align*}\label{12.41}
		\|R^*_{k,k'}(f)\|_{L^2(0,T;H^1(\Omega))} &\leq C 2^k \|R^*_{k,k'}(f)\|_{L^2(0,T;L^2(\Omega))}\\&\leq C |f|_{\mathfrak{B}} 2^{-k(s - 1 - \frac{d}{p} + \frac{d}{2}) - k'(\theta - \frac{1}{p'} + \frac{1}{2})}.
	\end{align*}
	Finally, since
	\[
	f - S^*_{k,k'}(f) = \sum_{j>k} \sum_{j'>k'} R^*_{j,j'}(f),
	\]
	for \( s - \frac{d}{p} + \frac{d}{2} > 1 \), \( \theta - \frac{1}{p'} + \frac{1}{2} > 0 \), the series converges in \( L^2(0,T;H^1(\Omega)) \).
	
	\textbf{Case (iii):}
	Let \( f \in \BV \) and \( \tilde f \) as in \eqref{12.23}, with \( f = \tilde f \) a.e. by Theorem~\ref{thm_12.2}. Consider local cubes \( J \times J' \subset \Omega_T \) of size \( \ell_J \leq rb, \ \ell_{J'} \leq r'b' \). Using approximation as in \eqref{12.22} and \eqref{12.23}, there exists \( P_{J,J'} \in \P_{r,r'} \) such that
	\[
	|\tilde f(x,t) - P_{J,J'}(x,t)| \leq C |f|_{\mathfrak{B}} b^{s - \frac{d}{p}} b'^{\theta - \frac{1}{p'}}.
	\]
	It follows that  
	\[
	|\Delta^r_h \Delta^{r'}_{h'} \tilde f(x,t)| \leq C |f|_{\mathfrak{B}} b^{s - \frac{d}{p}} b'^{\theta - \frac{1}{p'}},
	\]
	uniformly in \( (x,t) \), which proves continuity and \eqref{2.19}.
\end{proof}

\section{Optimal Recovery}\label{sec_3}
In this section, we examine whether the data \( (f, g, u_0) \) at fixed sites \( \X, \Y, \Z \) are sufficient to approximate \( u \) in the \( V \) norm within an accuracy \( \varepsilon \). This leads to an optimal recovery (OR) problem, with the answer depending on the number and placement of the data points, and the regularity of \( f, g, u_0 \). We assume that \( f, g, u_0 \) are continuous and lie in Besov spaces that compactly embed into spaces of continuous functions, ensuring \( u \in V \). Specifically, if \( U(\B) \) denotes the unit ball in the Besov space \( \B \), we assume
\begin{align*}
	u_0 \in \mathcal{M} &:= U(\B), \quad B = B^{ \check{s}}_{ \check{q}}(L^{ \check{p}}(\Omega)), \quad 1 \leq  \check{q},  \check{p} \leq \infty, \quad  \check{s} > d/ \check{p},\\
	f \in \F &:= U(\B), \quad \B = B^\theta_{p'q'}(0,T; B^s_{pq}(\Omega)),\ \ s > d/p, \, \theta > 1/p',\  p, q, p', q' \in [1,\infty],\\
	g \in \mathcal{G} &:=\text{Tr}(U(\B)), \quad B :=B^{\bar \theta}_{\bar p'\bar q'}( 0,T;B^{\bar{s}}_{\bar p\bar q}(\partial\Omega)),\ \bar{s} > d/\bar{p},\ \bar{\theta} > 1/\bar{p}', \ 1 \leq \bar{p}, \bar{q} \leq 2,
\end{align*}
with $1 \leq  \bar q', \bar p' \leq \infty $ so that $\M$ compactly embeds into $L^2(\Omega)$, \( \F \) compactly embeds into \( C(\Omega_T),\) and $\mathcal{G}$ as compact subset of $H^{1/2,1/4}(\Sigma)$ embeds into $C(\Sigma)$. 
According to these presumptions, the function $u$ that we wish to numerically approximate is a part of the set
\[
\mathcal{U} := \left\{ u \in C(\Omega) \, : \, u \text{ satisfies } \eqref{1.1} \text{ with } f \in \mathcal{F}, \, g \in \mathcal{G} \text{ and } u_0 \in\M \right\}.
\]
Given the data \( f := (f_1, \dots, f_{N_{\text{int}}}) \), \( g := (g_1, \dots, g_{N_{\text{sb}}}) \), and \( u_0 := (u_1, \dots, u_{N_{\text{tb}}}) \), the available information about the functions \( f \in \mathcal{F} \), \( g \in \mathcal{G} \), and \( u_0 \in \mathcal{M} \) is encoded in the corresponding data sets
\begin{align*}
	\begin{cases}
		\mathcal{F}_{\text{data}} &:= \left\{ f \in \mathcal{F} \, : \, f_i = \hat{f}_i, \, i = 1, \dots, N_{\text{int}} \right\}, \\
		\mathcal{G}_{\text{data}} &:= \left\{ g \in \mathcal{G} \, : \, g_i = \hat{g}_i = u(z_i), \, i = 1, \dots, N_{\text{sb}} \right\}, \\
		\mathcal{M}_{\text{data}} &:= \left\{ u_0 \in \mathcal{M} \, : \, u_{i,0} = \hat{u}_{i,0}, \, i = 1, \dots, N_{\text{tb}} \right\}.
	\end{cases}
\end{align*}
Thus,  the target function $u$ lies within the set
\begin{align*}
	\mathcal{U}_{\text{data}} &:= \left\{ u \in \mathcal{U} \, \middle| \, 
	\begin{aligned}
		u_t - \Delta u(y_i) &= \hat{f}_i, && i = 1, \dots, N_{int}, \\
		\hspace{-2mm}u(z_i) &= \hat{g}_i, && i = 1, \dots, N_{sb}, \\
		\hspace{-2mm}u(x_i) &= \hat{u}_{i,0}, && i = 1,\dots,N_{tb}.
	\end{aligned}
	\right\}\\
	&:= \left\{ u: u \text{ is a solution to } \eqref{1.1} \text{ with } f \in \mathcal{F}_{\text{data}}, \, g \in \mathcal{G}_{\text{data}}, u_0 \in \mathcal{M}_{\text{data}} \right\}.
\end{align*}
We aim to determine how accurately the condition \( u \in \mathcal{U}_{\text{data}} \) characterizes or determines the function \( u \).
The recovery rate depends not only on the data sites \( \mathcal{X}, \mathcal{Y}, \mathcal{Z} \), but also on the values \( (f, g, u_0) \) assigned at those sites. To obtain uniform estimates, we fix the data sites and define the \textit{uniform optimal recovery rate} as
\[
R^*(\mathcal{U}, \mathcal{X}, \mathcal{Y}, \mathcal{Z})_X := \sup_{u \in \mathcal{U}} R(\mathcal{U}_{\text{data}}(u))_X,
\]
which measures the worst-case error over the class \( \mathcal{U} \). If we prescribe a data budget \( m = |\mathcal{X}| + |\mathcal{Y}| + |\mathcal{Z}| \), the optimal recovery error under this constraint is
\[
R^*_m(\mathcal{U})_X := \inf_{\substack{\mathcal{X} \subset \Omega_T, \, \mathcal{G} \subset \Sigma,\, \mathcal{Z} \subset \Omega_0 \\ |\mathcal{X}| + |\mathcal{Y}| + |\mathcal{Z}| = m}} R^*(\mathcal{U}, \mathcal{X}, \mathcal{G}, \mathcal{Z})_X, \quad m \geq 2.
\]
Similarly, for recovering \( f \in \mathcal{F} \), we define
\[
R^*(\mathcal{F}, \mathcal{X})_X := \sup_{f \in \mathcal{F}} R(\mathcal{F}_{\text{data}}(f))_X, \quad 
R^*_m(\mathcal{F})_X := \inf_{\mathcal{X} \subset \Omega_T, \, |\mathcal{X}| = N_{int}} R^*(\mathcal{F}, \mathcal{X})_X.
\]
Analogous definitions apply for the recovery of \( g \) and \( u_0 \). All of these optimal recovery errors depend on the norm \( \| \cdot \|_X \) used to measure the error.

\subsection{Optimal recovery of $f$} 
Let \( \F := U(\B) \), where \( \B = B^{\theta}_{p'q'}(0,T;B^s_{pq}(\Omega)) \) with \( 1 \leq p, q \leq \infty \), \( s > d/p \), and \( \theta > 1/p' \). Define the tensor product grid \( G^{r,r'}_{k,k'} \subset \bar{\Omega}_T = [0,1]^d \times [0,T] \) for \( r, r' \geq 2 \) as
\begin{align}\label{3.8}
	G^{r,r'}_{k,k'}:= \{y_1, \cdots , y_m\}\subset \bar \Omega_T, m = (r2^k)^d(r'2^{k'}),
\end{align}  
with spatial and temporal spacing \( h = 2^{-k}(r-1)^{-1} \), \( h' = 2^{-k'}(r'-1)^{-1} \), respectively gives  the optimal rate \( R^*(\F, \X)_X \asymp m^{-\alpha_X} \). 

\begin{theorem}\label{thm_3.1}
	Let \( \Omega_T = (0,1)^d \times (0,T] \) and \( \F := U(\B) \), where \( \B\) is any Besov space \(B^{\theta}_{p'q'}(0,T; B^s_{pq}(\Omega)) \) with \( s > d/p,\ \theta > 1/p' \), and \( 1 \leq p, q, p', q' \leq \infty \). Then for \( m = \tilde{m} \times \hat{m} \), the optimal recovery rate \( R^*_m(\F)_X \) satisfies:
	\begin{enumerate}[(i)]
		\item For \( X = C(\Omega_T) \): 
		\begin{align}\label{3.9}
			R^*_m(\F)_X \asymp \tilde{m}^{-\alpha_C} \hat{m}^{-\beta_C},\quad \alpha_C := \frac{s}{d} - \frac{1}{p},\ \beta_C := \theta - \frac{1}{p}.
		\end{align}
		\item For \( X = L^{\tau'}(0,T; L^\tau(\Omega)) \), \( \tau, \tau' \geq 1 \): 
		\begin{align}\label{3.10}
			R^*_m(\F)_X \asymp \tilde{m}^{-\alpha_\tau} \hat{m}^{-\beta_\tau},\quad \alpha_\tau := \frac{s}{d} - \left[\frac{1}{p} - \frac{1}{\tau} \right]_+,\ \beta_\tau := \theta - \left[\frac{1}{p} - \frac{1}{\tau'} \right]_+.
		\end{align}		
		\item For \( X = L^2(0,T; H^1(\Omega)) \), \( 1 \leq p \leq 2 \) and \(1\leq p' \leq \infty\): 
		\begin{align}\label{3.11}
			R^*_m(\F)_X \asymp \tilde{m}^{-\alpha_H} \hat{m}^{-\beta_H},\ \alpha_H := \frac{s - 1}{d} - \left[\frac{1}{p} - \frac{1}{2} \right],\, \beta_H := \theta - \left[\frac{1}{p} - \frac{1}{2} \right].
		\end{align}
		\item For \( X = L^2(0,T; H^{-1}(\Omega)) \), \( 1 < p, p' \leq \infty \): 
		\begin{align}\label{3.12}
			R^*_m(\F)_X \asymp \tilde{m}^{-\alpha_{-1}} \hat{m}^{-\beta_H},\quad \alpha_{-1} := \frac{s}{d} - \left[\frac{1}{p} - \frac{1}{\delta} \right]_+,\ \frac{1}{\delta} := \frac{1}{2} + \frac{1}{d}.
		\end{align}
	\end{enumerate}
	The constants of equivalence are independent of \( m \).
\end{theorem}
%
\begin{proof}
	We present the proof for \( d \geq 2 \) for simplicity, it is divided into two parts, upper and lower bound.
	
	\textbf{Proof of upper bounds:} 
	It is sufficient to prove the upper bounds for \( \V = B^\theta_{p'}(0,T;B^s_p(\Omega))\). We construct \( m \) data sites as tensor grids \( G^{r,r'}_{k,k'} \subset \bar\Omega_T \). For other \( m \), monotonicity of \( R^*_m \) ensures the bounds hold.\\
	\textbf{Cases (i) and (ii):} Using the simplicial partition \( \T_{k,k'}(\Omega_T) \), we build from the data \( (f_i) \) a continuous piecewise polynomial approximant, as shown previously.
	\begin{align}\label{12.42}
		S = S^*_{k,k'}(f) = \sum_{E\times I'\in \mathcal{T}_{k,k'}}^{}L_{E,I'}(f)\chi_T\chi_{I'} = \sum_{I'\in \I_{k'}}^{}\sum_{I\in \D_{k}}^{}L_{I,I'}(f)\chi_I\chi_{I'}.
	\end{align}
	Let \( L_{E,I'}(f) \in P_{r,r'} \) be interpolant of \( f \) on \( (E\times I') \cap G_{k,k'} \). As shown in \eqref{2.17}, for \( \tau \geq p \),
	\[\|f - S\|_{L^{\tau'}(0,T;L^\tau(\Omega))} \leq |f|_{\V}\, 2^{-(kd\alpha_\tau + k'\beta_\tau)} \leq |f|_{\V}\, \tilde m^{-\alpha_\tau} \hat m^{-\beta_\tau}.\]
	This also holds for \( \tau < p \) by embedding results. Since \( S \) depends only on \( (f_i) \), every \( f \in \F_{\text{data}} \) lies at distance \( C \tilde m^{-\alpha_\tau} \hat m^{-\beta_\tau} \) of \( S \), follows the proof for cases (i) and (ii).\\
	\textbf{Case (iii):} This is similar to case (ii) except now we use theorem \eqref{$thm_2.2$}.\\
	\textbf{Case (iv):} We shall need the following lemma.
	\begin{lemma}\label{lemma_12.4}
		Let \( \OMT = (0,1)^d \times (0,T] \), with \( d \geq 3 \), and let \( \gamma = \delta(d) \) as in \eqref{3.12}. Then
		\[
		\|\tilde f\|_{L^2(0,T;H^{-1}(\Omega))} \leq C \|\tilde f\|_{L^2(0,T;L^\gamma(\Omega))},
		\]
		for all \( \tilde f \in L^2(0,T;L^\gamma(\Omega)) \), where \( C \) depends only on \( d \). For \( d = 2 \) and \( 1 < \tau \leq \infty \), we have
		\[
		\|\tilde f\|_{L^2(0,T;H^{-1}(\Omega))} \leq C \frac{\tau}{\tau - 1} \|\tilde f\|_{L^2(0,T;L^{\tau}(\Omega))}.
		\]
	\end{lemma}
	
	\begin{proof}
	Proof will follow with an integration over [0,T] from, \cite[ Lemma~12.4]{BVPS}.
	\end{proof}
	We now prove the upper bound for case (iv). Since $\alpha_{-1}$ remains constant for $p \geq \gamma$, it suffices to consider $p \leq \gamma$. For $p > \gamma$, the set $U(B^\theta_{p'q'}(0,T;B^s_{pq}(\Omega)))$ is contained in $U(B^\theta_{p'\infty}(0,T;B^s_{\gamma\infty}(\Omega)))$, so the upper bound reduces to the case $\mathcal{F} = U(B^\theta_{p'q'}(0,T;B^s_{pq}(\Omega)))$ with $p \leq \gamma$. For $p \leq \gamma$, we have $U(B^\theta_{p'q'}(0,T;B^s_{pq}(\Omega)))$ $\subset U(B^\theta_{p'\infty}(0,T;B^s_{p\infty}(\Omega)))$. Thus, shows only for $\mathcal{F} = U(\V)$ with $p \leq \gamma$.

	Let \( f \in \mathcal{F} := U(\V) \) with \( s > d \), \( p \leq \gamma \), and \( p' \leq 2 \). Since \( f \in C(\Omega_T) \), we define \( S \) from \eqref{12.42}. From \eqref{2.17}  and Lemma \ref{lemma_12.4} we obtain
	\begin{align*}\label{12.50}
		\|f-S\|_{L^2(0,T;H^{-1}(\Omega))}\leq C\|f - S\|_{L^2(0,T;L^\gamma(\Omega))}
		\leq|f|_{\V}\, 2^{-(kd\alpha_{-1}+k'\beta_{-1})}.
	\end{align*}
	Since \( S \) depends only on the data, every \( g \in \mathcal{F}_{\text{data}}(f) \) satisfies \(\|g-S\|_{L^2(0,T;H^{-1}(\Omega))} \leq C \tilde{m}^{-\alpha_{-1}}\hat{m}^{-\beta_{-1}}\), bounding the recovery rate by \( C m^{-\alpha_{-1}}\hat{m}^{-\beta_{-1}} \). Then \( f \) extends this to \( R^*(\mathcal{F}) \), completing the upper bound proof for case (iv) .

	\textbf{Proof of lower bounds: } To prove the lower bounds in Theorem~\ref{thm_3.1}, we construct a function \( \eta \in \F \) such that  \( \eta(x_i,t_j) = 0 \  i=1,\dots,\tilde{m},\ j=1,\dots,\hat{m}\) and \( \|\eta\|_X \geq c \tilde{m}^{-\alpha_\tau} \hat{m}^{-\beta_\tau} \). Since both \( \eta \) and the zero function satisfy the same data, this yields the desired lower bound for \( R^*(\F)_X \). To construct such an \( \eta \), let \( \varphi \in C^\infty(\mathbb{R}^d \times \mathbb{R}^+) \), non-negative, supported in \( \Omega \times (0,T] \), and satisfying
	\begin{align}\label{12.57}
		\|\varphi\|_{L^\infty} = 1, \quad \varphi(x,t) \geq \frac{1}{2} \text{ on } \Omega_0 \times (T/4, 3T/4), \quad \text{with} \  \Omega_0 := [1/4, 3/4]^d.
	\end{align}
	
	We take \( \varphi(x,t) = \phi(x_1)\cdots\phi(x_d)\phi(t) \), where \( \phi \) is a suitable univariate function, and choose \( \varphi \) depends on  \( s, \theta, p, q, p', q' \) to minimize the norm
	\begin{align}\label{12.58}
		M := \|\varphi\|_{B^\theta_{p'q'}(0,T;B^s_{pq}(\Omega))}.
	\end{align}
	For any cuboid \( I \times I' \subset \Omega \times (0,T] \), let \( \xi_I, \xi_{I'} \) be its smallest vertices and \( \ell_I, \ell_{I'} \) the side lengths. Define
	\begin{align}\label{12.59}
		\varphi_{I,I'}(x,t) := \ell_I^{s - \frac{d}{p}}\ell_{I'}^{\theta - \frac{1}{p'}} \varphi\left(\frac{x - \xi_I}{\ell_I}, \frac{t - \xi_{I'}}{\ell_{I'}}\right),
	\end{align}
	supported in \( I \times I' \), vanishing on its boundary, and satisfying
	\begin{align}\label{12.60}
		\|\varphi_{I,I'}\|_{B^\theta_{p'q'}(0,T;B^s_{pq}(\Omega))} = M.
	\end{align}
	To lower-bound \( R^*_m(\F)_X \), we may assume \( m = 2^{kd}2^{k'} \). Given any \( m \) samples \( S_m \subset \Omega \times (0,T] \), construct a tensor grid \( G^{2,2}_{k+2,k'+2} \) with spacing \( 2^{-(k+2)} \) in space and \( 2^{-(k'+2)} \) in time. This grid has \( 4^{d+1}m \) cuboids, at least \( C_d m = (4^{(d+1)} - 1)m \) of them, denoted \( \mathcal{I}(S_m) \), contain no samples and have volume \( \gtrsim m^{-1} \). For each \( I \times I' \in \mathcal{I}(S_m) \), define $	\eta_{I,I'} := M^{-1} \varphi_{I,I'}.$
	Each function \( \eta_{I,I'} \in \mathcal{F} \) and vanishes at all data sites. These functions \( \eta_{I,I'} \) will be used to establish the lower bounds in Theorem~\ref{thm_3.1}.
	\medskip
	
	\textbf{Case (i):} Let us fix \( s, \theta, p, q, p', q' \), take any cuboid \( I \times I' \in \mathcal{I} \) and define \( \eta := \eta_{I,I'} \). Then \( \eta \) vanishes at all data sites. Thus we get
	\[
	\|\eta\|_{L^\infty(\Omega_T)} = M^{-1} \ell_I^{s - \frac{d}{p}} \ell_{I'}^{\theta - \frac{1}{p'}} \geq c \tilde m^{-\alpha_C} \hat m^{-\beta_C}.
	\]
	Hence, \( R^*(\mathcal{F})_{C(\Omega_T)} \geq c \tilde m^{-\alpha_C} \hat m^{-\beta_C} \), proves the lower bound in case (i).
	
	\medskip
	
	\textbf{Case (ii):} We consider two subcases based on values of \( \tau, \tau' \).
	
	\textit{Case \( \tau \geq p \), \( \tau' \geq p' \):}  
	Let \( \eta = \eta_{I,I'} \) for any fixed cuboid \( I \times I' \in \mathcal{I} \). Then, using the scaling properties of \( \varphi_{I,I'} \),
	\begin{align*}
		\|\eta\|_{L^{\tau'}(0,T;L^\tau(\Omega))} &= M^{-1} \|\varphi_{I,I'}\|_{L^{\tau'}(0,T;L^\tau(\Omega))} \geq c M^{-1} \ell_I^{s - \frac{d}{p} + \frac{d}{\tau}} \ell_{I'}^{\theta - \frac{1}{p'} + \frac{1}{\tau'}} \\&\geq c \tilde m^{-\left( \frac{s}{d} - \frac{1}{p} + \frac{1}{\tau} \right)}\hat{m}^{-\left( \theta- \frac{1}{p'} + \frac{1}{\tau'} \right)}= c \tilde m^{-\alpha_\tau}\hat{m}^{-\beta_{\tau'}}.\\
	\end{align*}
	Hence, \( R^*(\F)_{L^{\tau'}(0,T;L^\tau(\Omega))} \geq c \tilde m^{-\alpha_\tau} \hat m^{-\beta_{\tau'}} \).
	
	\textit{Case \( \tau < p < \infty \), \( \tau' < p' < \infty \):}  
	It suffices to prove the bound for \( p = p' = \infty \). Define $\eta := \kappa \sum_{I \times I' \in \mathcal{I}} \eta_{I,I'},$ with \( \kappa \) chosen so that \( \eta \in U(B^\theta_{\infty q'}(0,T;B^s_{\infty q}(\Omega))) \). Each \( \eta_{I,I'} \) has disjoint support and satisfies $\|\eta\|_{L^\infty(\Omega_T)} \leq M 2^{-ks} 2^{-k'\theta}.$ Furthermore, for \( r := \lceil s \rceil + 1, \ r' := \lceil \theta \rceil + 1 \), the modulus of smoothness of $\eta$ satisfies
	\begin{equation}\label{12.63}
		\omega_{r,r'}(\eta, b\times b')_\infty \leq M \kappa 2^{-ks-k'\theta} \min(1, 2^{kr} b^r) \min(1, 2^{k'r'} b'^r).
	\end{equation}
	Following \cite{BVPS}, we estimate the Besov norm via integral splitting, which gives
	\begin{align}\label{12.64}
		\int_0^T \left( b'^{-\theta} \omega_{r,r'}(\eta, b\times b')_{B^s_{\infty q}(\Omega)} \right)^{q'} \frac{db'}{b}
		\leq \kappa^{q'} C^{q'}, 
	\end{align}
	for constant \( C \), implying \( \kappa \geq c \). Since \( \eta \geq \kappa 2^{-ks - k'\theta} \) on a set of positive measure, and $\|\eta\|_{L^{\tau'}(0,T;L^\tau(\Omega))} \geq c 2^{-ks - k'\theta} \geq c \tilde{m}^{-s/d} \hat{m}^{-\theta},$ which completes the proof for case (ii).
	
	\medskip
	
	\textbf{Case (iii):}  
	Assume \( p, p' \leq 2 \) and take \( \eta = \eta_{I,I'} \).Then
	\begin{align}\label{12.65}
		\|\eta\|_{L^2(0,T;H^1(\Omega))} &= M^{-1} \|\varphi_{I,I'}\|_{L^2(0,T;H^1(\Omega))} \geq c M^{-1} \ell_I^{s - 1 - \frac{d}{p} + \frac{d}{2}}\ell_{I'}^{\theta - \frac{1}{p'} + \frac{1}{2}} \notag \\&\geq c \tilde m^{-\left( \frac{s}{d} - \frac{1}{d} - \frac{1}{p} + \frac{1}{2} \right)}\hat m^{-\left( \theta - \frac{1}{p'} + \frac{1}{2} \right)}, 
	\end{align}
	which gives the desired lower bound for \( R^*(\mathcal{F})_{L^2(0,T;H^1(\Omega))} \) for any fixed \( I \times I' \in \mathcal{I} \).
	
	\medskip
	
	\textbf{Case (iv):}  
	Assume \( p \leq \gamma \). For any \( I \times I' \in \mathcal{I} \), define \( \eta = \eta_{I,I'} \in \mathcal{F} \), which vanishes at all data sites and satisfies
	\begin{align}\label{12.66}
		\eta(x,t) \geq c \tilde m^{- \left( \frac{s}{d} + \frac{1}{p} \right)}\hat m^{- \left( \theta + \frac{1}{p'} \right)}, \quad (x,t) \in I_0\times I_0',
	\end{align}
	where $I_0\times I_0' = [\xi_I - \ell_I /4, \xi_I + \ell_I /4]^d\times [\xi_{I'} - \ell_{I'} /4, \xi_{I'} + \ell_{I'} /4]$. To estimate \( \|\eta\|_{L^2(0,T;H^{-1}(\Omega))} \), define test function \( v \in L^2(0,T;H_0^1(\Omega)) \) supported in \( I \times I' \) by
	\begin{align}\label{12.67}
		v(x,t) := c \ell_I^{1 - d/2} \ell_{I'}^{-1/2}  \, \varphi(\ell_I^{-1}(x - \xi_I), \ell_{I'}^{-1}(t - \xi_{I'})),
	\end{align}
	with \( \|v\|_{L^2(0,T;H^{-1}(\Omega))} = 1 \), and satisfying
	\begin{align}\label{12.68}
		v(x,t) \geq c \ell_I^{1 - d/2}&\ell_{I'}^{-1/2} \geq c \tilde m^{-1/d + 1/2}\hat{m}^{1/2}, \quad (x,t) \in I_0\times I_0'.
	\end{align}
	Then,
	\begin{align} \label{12.69}
				\|\eta\|_{L^2(0,T;H^{-1}(\Omega))} &\geq \int_{0}^{T}\int_{\Omega} \eta(x,t) v(x,t) \, dx\, dt \geq  \int_{I_0'}\int_{I_0} \eta(x,t) v(x,t) \, dx\,dt \notag\\
		&\geq c \tilde m^{-s/d + 1/p} \tilde m^{-1/d + 1/2}\hat{m}^{-\theta+1/p'+1/2} |I_0|\, |I_0'|,
	\end{align}
	which proves the desired lower bound for \( R^*(\mathcal{F})_{L^2(0,T;H^{-1}(\Omega))} \).
	
	For the case \( p > \gamma \), let \( \eta \in \mathcal{F} = U(B^\theta_{p'q'}(0,T;B^s_{\infty q}(\Omega))) \) vanish at data sites. Define \( v = c\varphi \) such that \( \|v\|_{L^2(0,T;H^{-1}(\Omega))} = 1 \). Since \( v \geq c \) on \( \Omega_0 \times [T/4, 3T/4] \), and \( \eta \geq 0 \), we compute
	\begin{align}\label{12.71}
		\|\eta\|_{L^2(0,T;H^{-1}(\Omega))} &\geq \int_{0}^{T}\int_{\Omega} \eta(x,t) v(x,t) \, dx\,dt \geq \int_{T/4}^{3T/4}\int_{\Omega_0} \eta(x,t) v(x,t) \, dx\,dt \notag\\
		&\geq \int_{T/4}^{3T/4}\int_{\Omega_0} \eta(x,t) \, dx\,dt\geq \kappa \sum_{I\times I' \in \mathcal{I}_0} \int_{I'}\int_{I} \eta_{I,I'}(x,t) \, dx\,dt\notag\\ 
		&\geq c 2^{-ks} 2^{-kd}2^{-k'\theta}2^{-k'} \#(\mathcal{I}_0), 
	\end{align}
	where $\mathcal{I}_0$ is the set of $I\times I' \in \mathcal{I}$ with $I\times I' \subset \Omega_0\times [T/4,3T/4]$. The number of such cuboids are at least
	\[
	\#(\mathcal{I}_0) \geq 2^{(d+1)} m - m = (2^{(d+1)} - 1)m,
	\]
	with $m = 2^{kd}2^{k'}$. Substituting into \eqref{12.71} gives
	\begin{align}
		\|\eta\|_{L^2(0,T;H^{-1}(\Omega))} \geq c \tilde m^{-s/d}\hat{m}^{-\theta},
	\end{align}
	which implies $R^*(\mathcal{F})_{L^2(0,T;H^{-1}(\Omega))} \geq c \tilde m^{-\alpha_{-1}}\hat{m}^{-\beta_{-1}}$ when $p > \gamma$, since $1/p - 1/\gamma < 0$. This concludes the proof of lower bound in (iv), and thus the proof of Theorem \ref{thm_3.1}.
\end{proof}
\subsection{Optimal recovery of $g$}
In this section, we analyze the optimal recovery rate of functions \( g \in \mathcal{G} \) in the norm \( H^{1/2,1/4}(\Sigma) \). Let \( \text{Tr} = \text{Tr}_{\Sigma} \) be the trace operator from \( \Omega_T \) onto \( \Sigma \), and define the model class
\[
\mathcal{G} := \{ g = \text{Tr}(v) : \|v\|_{\overline \B} \leq 1 \}, \quad \overline{\B} := B^{\bar\theta}_{\bar{p}'\bar{q}'} (0,T; B^{\bar{s}}_{\bar{p} \bar{q}}(\Omega)),
\]
where \( \bar{s} > d/\bar{p},\, \bar{\theta}>1/\bar{p}' \), ensuring continuity and compact embedding into \( V \) for \( 1\leq\bar{p}, \bar{p}'\leq 2 \) and \( d\geq2 \). We define \( \overline{G}^{r,r'}_{k,k'} := G^{r,r'}_{k,k'} \cap \Sigma \), the lateral boundary grid of size \( \bar{m}\times\hat{m} = 2d [r 2^k]^{d - 1}[r'2^{k'}] \asymp 2^{k(d-1)}2^{k'} \). For recovery, we construct the interpolant
\begin{align}\label{3.18}
	\overline S_{k,k'}(g) := \sum_{j=1}^{\hat{m}} \sum_{i=1}^{\bar{m}} g(x_i,t_j) \bar{\phi}_i(x)\bar{\phi}_j(t),
\end{align}
where \( \bar{\phi}_i\bar{\phi}_j \) are traces of the Lagrange elements \( \phi_i\phi_j \). If \( v \in \overline{\B} \) with \( \text{Tr}(v) = g \), then \( \text{Tr}(S^*_{k,k'}(v)) = \overline{S}_{k,k'}(g) \). 

\begin{theorem}\label{thm_3.2}
	Let \( \overline{\B} = B^{\bar\theta}_{\bar{p'}\bar{q'}}(0,T; B^{\bar{s}}_{\bar{p}\bar{q}}(\Omega)) \) with \( \bar{s} > d/\bar{p} \), \( 1 \leq \bar{p}, \bar{p}' \leq 2 \), \( 1 \leq \bar{q}, \bar{q}' \leq \infty \), and define the model class $	\mathcal{G} := \{ \text{Tr}(v) : \|v\|_{\overline{\B}} \leq 1 \}.$ Then the optimal recovery rate in \( H^{1/2,1/4}(\Sigma) \) satisfies
	\begin{align}\label{3.19}
	R_m^*(\mathcal{G})_{H^{1/2,1/4}(\Sigma)} \asymp \bar{m}^{-\bar\alpha} \hat{m}^{-\beta_H}, \quad \bar m, \hat{m} \geq 1, 
	\end{align}
	with \( \bar\alpha := \frac{\bar{s} - 1}{d - 1} - \frac{d}{d - 1} \left( \frac{1}{\bar{p}} - \frac{1}{2} \right)\) and \( \beta_H \) as in \eqref{3.11}, and constants independent of \( \bar{m}, \hat{m} \). The rate, achieved by using \( \Y = \overline{G}^{r,r'}_{k,k'} \), provided \( r > \max(\bar{s},1) \), \( r' > \max(\bar{\theta},1) \), \(k\geq k'\) and approximating \( g \) from \( S_{k,k'}(g) \).
\end{theorem}
\begin{proof}
	To establish the upper bound in \eqref{3.19}, we start with the trace norm inequality
	\begin{align}
		\|g-\bar{S}_{k,k'}\|_{H^{1/2,1/4}(\Sigma)}\leq \|v-S_{k,k'}\|_{L^2(0,T;H^1(\Omega))} + \|v-S_{k,k'}\|_{H^1(0,T;H^{-1}(\Omega))}
	\end{align}
	We first estimate the error in the \( L^2(0,T;H^1(\Omega)) \) norm
	\begin{align}\label{12.73}
		\|v-S_{k,k'}\|_{L^2(0,T;H^{1}(\Omega))}&\leq C[2^{kd}]^{-\alpha_H} [2^{k'}]^{-\beta_H}\notag\\
		&\leq C\,2^{-kd\left(\frac{\bar s-1}{d}-\left(\frac{1}{\bar p}-\frac{1}{2}\right)\right)}2^{-k'\left(\bar\theta - \left(\frac{1}{\bar p'}-\frac{1}{2}\right)\right)} \notag\\& \leq C\,2^{-k(d-1)\left(\frac{\bar s-1}{d-1}-\frac{d}{(d-1)}\left(\frac{1}{\bar p}-\frac{1}{2}\right)\right)}2^{-k'\left(\bar \theta - \left(\frac{1}{\bar p'}-\frac{1}{2}\right)\right)} \leq C\,\bar m^{-\bar\alpha}\hat{m}^{-\beta_H}.	
	\end{align}
	Next, we consider \( H^1(0,T;H^{-1}(\Omega)) \) norm. Using the approach as in Theorem \ref{$thm_2.2$}, for $k\geq k'$
	\begin{align}
		\|v-S_{k,k'}\|_{H^1(0,T;H^{-1}(\Omega))}&\leq C\,2^{-kd\left(\frac{\bar s}{d}-\left(\frac{1}{\bar p}-\frac{1}{\delta}\right)\right)} 2^{-k'\left(\bar \theta - 1- \left(\frac{1}{\bar p'}-\frac{1}{2}\right)\right)}\notag\\
		&\leq C\,2^{-k(d-1)\left(\frac{\bar s-1}{d-1}-\frac{d}{(d-1)}\left(\frac{1}{\bar p}-\frac{1}{\delta}\right)\right)}2^{-k'\left(\bar \theta - \left(\frac{1}{\bar p'} -\frac{1}{2}\right)\right)} \leq  C\,\bar m^{-\check\alpha}\hat{m}^{-\beta_H},	
	\end{align}
    where \( \check{\alpha} = \frac{\bar s-1}{d-1} - \frac{d}{d-1} \left( \frac{1}{\bar p} - \frac{1}{\delta} \right) \). Combining both estimates, we obtain
	\begin{align}\label{12.74}
		R_m(g)_{H^{1/2,1/4}(\Sigma)}\leq C\, \{\bar m^{-\bar\alpha}\hat{m}^{-\beta_H} + \bar m^{-\check\alpha} \hat{m}^ {-\beta_H}\} \leq C\,\bar m^{-\bar\alpha}\hat{m}^{-\beta_H},
	\end{align}
	where \( \bar{m} = 2d[(r-1)2^k]^{d-1} \) and \( \hat{m} = 2^{k'} \). Since the choice of \( g \in \mathcal{G} \) was arbitrary, it follows $	R^*_m(\G)_{H^{1/2,1/4}(\Sigma)}\leq C\, \bar m^{-\bar\alpha}\hat{m}^{-\beta_H}.$ By monotonicity of \( R_m^* \), this estimate holds for all \( m \), thereby completing the proof of the upper bound in \eqref{3.19}. We now prove the lower bound in \eqref{3.19}. Due to the monotonicity of \( R^*_m(\G) \) in \( H^{1/2,1/4}(\Sigma) \), it suffices to show the inequality for \( m = 2^{k'}(2^{k(d-1)} - 1) \), for any non-negative integer \( k \), using a similar argument to Theorem \ref{thm_3.1}.
	
	Let \( \Y := \{(x_i, t_j)\} \subset \partial\Omega \times (0,T] \) be data sites with \( \bar m = 2^{k(d-1)} - 1 \) and \( \hat m = 2^{k'} \), and consider the face  \( F := \{ (x,t) \in \Omega_T : x \cdot e_1 = 0, e_1 = (1, 0, \ldots, 0) \in \mathbb{R}^d\} \). Among the dyadic cubes in \( \mathcal{D}_{k,k'}(F) \), there exist a cube \( \bar J \times \bar J' \) contains no data sites from \( \Y \). Let \( J \times J' \in \mathcal{D}_{k,k'}(\Omega_T) \) be the cube having \( \bar J \times \bar J' \) as a face. If \((\xi'_{J}, \xi'_{J'}) \) is the center of \( J\times J'\), then the point \( (\xi'_{J}, \xi'_{J'}) - (2^{-k-1}, 0, \ldots, 0) \) is the center of \( \bar J\times \bar{J'} \). Define
	\begin{align}\label{12.76}
		v(x,t) := M^{-1} \phi_{J,J'}\left((x,t) - (2^{-k-1}, 0, \ldots, 0)\right),
	\end{align}
	and set $\eta := T_{\Sigma} v,$ so that \( \eta \in \G \) and vanishes at all data sites \( \Y \). We now show that $\|\eta\|_{H^{1/2,1/4}(\Sigma)} \geq c \bar m^{-\bar\alpha}\hat{m}^{-\beta_H}$ . Let \( \eta_0 \) denote the trace of \( \phi \) on \( x \cdot e_1 = 1/2 \), define $	\bar M := |\eta_0|_{L^2(0,T;H^{1/2}(\partial\Omega))} > 0, \ \tilde M := |\eta_0|_{H^{1/4}(0,T;L^2(\partial\Omega))} >0.$ Then
	\begin{align}\label{12.80}
		|\eta|_{L^2(0,T;H^{1/2}(\partial\Omega))} &= \bar M M^{-1} 2^{-k(\bar{s} - d/\bar{p})} 2^{k/2} 2^{-k(d-1)/2} 2^{-k'(\bar{\theta}-1/\bar{p}')}2^{-k'/2}\\\notag
		&= \bar M M^{-1}2^{-k(d-1) \left(\frac{\bar{s}-1}{d-1}-\frac{d}{(d-1)}\left(\frac{1}{\bar p}-\frac{1}{2}\right)\right)} 2^{-k'\left(\theta-\left (\frac{1}{\bar p'}-\frac{1}{2}\right)\right)}\geq c \bar m^{-\bar\alpha}\hat{m}^{-\beta_H}.
	\end{align}
	Similarly,
	\begin{align}
		|\eta|_{H^{1/4}(0,T;L^2(\partial\Omega))} &= \tilde M M^{-1} 2^{-k(\bar{s} - d/\bar{p})} 2^{-k(d-1)/2} 2^{-k'(\bar{\theta} -1/\bar{p'})} 2^{k'/4}2^{-k'/2}\\ \notag
		&= \tilde M M^{-1} 2^{-k(d-1)\left(\frac{\bar s-1}{d-1}-\frac{d}{d-1}\left(\frac{1}{\bar p}-\frac{1}{2}\right)\right)} 2^{-\frac{k}{2}} 2^{-k' \left( \bar \theta-\left(\frac{1}{\bar p'}-\frac{1}{2}\right)\right)}2^{-\frac{k'}{4}}\geq c\bar{m}^{-\bar\alpha-\frac{1}{2}} \hat{m}^{-\beta_H-\frac{1}{4}}.
	\end{align}
	Combining both components of the \( H^{1/2,1/4}(\Sigma) \) norm gives the lower bound.
	\begin{align}
		R^*_m(\G)_{H^{1/2,1/4}(\Sigma)}&\geq c\, \{\bar m^{-\bar\alpha}\hat{m}^{-\beta_H} + \bar m^{-\bar\alpha}\hat{m}^{-\beta_H}\bar{m}^{-\frac{1}{2}}\hat{m}^{-\frac{1}{4}}\}\\ &\geq c\,\bar m^{-\bar\alpha}\hat{m}^{-\beta_H}.\notag
	\end{align}
\end{proof}
\subsection{Optimal recovery of $u_0$}
Let \( \mathcal{M} := U(\B_0) \), with \( \B_0 = B^{\check{s}}_{p q}(\Omega) \), \( \check{s} > d/\check{p} \). For \( r \geq 2 \), define the initial grid \( G_{k,r} = \{(x_i,0)\}_{i=1}^{\tilde{m}} \subset \Omega \times \{0\} \), where \( \tilde{m} = (r 2^k)^d \). The optimal recovery rate in norm \( X \) then satisfies $R^*_{\tilde{m}}(\mathcal{M})_X \asymp \tilde{m}^{-\alpha_X},$
with the upper bound achieved on \( G_{k,r} \), and the lower bound holding for any set of \( \tilde{m} \) data points. See \cite[Section~3.1]{BVPS} for the proof.
\begin{theorem}\label{thm_u0}
	Let \( \Omega = (0,1)^d \), and \( \M := U(\B_0) \) with \( \B_0 = B^{\check{s}}_{\check{p}\check{q}}(\Omega) \), where \( \check{s} > d/\check{p} \), \( 1 \leq \check{p} \leq 2 \), \( 1 < \check{q} \leq \infty \). Then the optimal recovery rate in \( L^2(\Omega) \) satisfies
	\begin{align}\label{u_0}
		R^*_{\tilde{m}}(\M)_{L^2(\Omega)} \asymp \tilde{m}^{-\alpha}, \quad \tilde m \geq 1, \quad \alpha := \frac{\check{s}}{d} - \left( \frac{1}{\check{p}} - \frac{1}{2} \right),
	\end{align}
	with constants of equivalence independent of \( \tilde m \).
\end{theorem}

\subsection{Optimal recovery of $u$}
We are now going to study the optimal recovery rate in \( V := \{v \in L^2(0,T;H^1(\Omega)) : v_t \in L^2(0,T;H^{-1}(\Omega))\} \) for functions \( u \in \U \), given data satisfies \eqref{1.1}
under budget \( m = (\tilde{m} + \bar{m})\hat{m} \). For ,  \(s > d/p,\quad \theta > 1/p',\quad \bar{s} > d/\bar{p},\quad \bar{\theta} > 1/\bar{p}',\quad \check{s} > d/\check{p}\) and \( \F := U(B^\theta_{p'q'}(0,T;B^s_{pq}(\Omega))) \), \( \G := \operatorname{Tr}(U(B^{\bar{\theta}}_{\bar{p}'\bar{q}'}(0,T;B^{\bar{s}}_{\bar{p}\bar{q}}(\Omega)))) \), and \( \M := U(B^{\check{s}}_{\check{p}\check{q}}(\Omega),)\) the model class for $u$ is
\[
\U := \left\{ u \in C(\Omega) : u \text{ solves \eqref{1.1} with } f \in \F,\, g \in \G,\, u_0 \in \M \right\}.
\]

The following theorem gives the optimal recovery rate for \( \U \).
\begin{theorem}
	Let \( \Omega_T = (0,1)^d \times (0,T] \), and let \( \F, \G, \M, \U \) be defined as above. Then the optimal recovery rate in \( V \) satisfies $R^*_m(\U)_V \asymp m^{-\min(a,b,c)}.$
\end{theorem}
\begin{proof}
	Let data sites \( \X, \Y, \Z \) be fixed with \( |\X| = \tilde m \hat{m} \), \( |\Y| = \bar m \hat{m} \), \( |\Z| = \tilde m \), and total budget \( m = (\tilde m + \bar m)\hat{m} + \tilde{m} \). For any \( u_1, u_2 \in \U_{\text{data}} \) with corresponding \( f_i \in \F_{\text{data}},\, g_i \in \G_{\text{data}},\, u_0^i \in \M_{\text{data}} \), it follows from \eqref{1.3} that
	\begin{align*}\label{3.23}
		\|u_1 - u_2\|_{V} \asymp \|f_1 - f_2\|_{L^2(0,T;H^{-1}(\Omega))} + \|g_1 - g_2\|_{H^{1/2,1/4}(\Sigma)} + \|u_0^1-u_0^2\|_{L^2(\Omega)},\notag\\
		R^*(\U, \X, \Y, \Z)_{V} \asymp R^*(\F, \X)_{L^2(0,T;H^{-1}(\Omega)))} + R^*(\G, \Y)_{H^{1/2,1/4}(\Sigma)} + R^*(\M, \Z)_{L^2(\Omega)}.
	\end{align*}
	%
	Applying Theorems~\ref{thm_3.1}, \ref{thm_3.2}, and \ref{thm_u0}, and assuming \( \tilde m \hat{m} \asymp \bar m \hat{m} \asymp m/2 \), we obtain the upper bound
	\[
	R^*_m(U)_{V} \lesssim m^{-a} + m^{-b} + m^{-c} \lesssim m^{-\min(a, b, c)}. 
	\]	
	Conversely, applying the same theorems, we obtain
	\begin{align*}
		m^{-\min{(a,b,c)}}&\lesssim \inf_{m = \tilde{m}\times \hat{m} + \tilde{m}\times\hat{m} + \tilde{m}} \left((\tilde{m}\times \hat{m})^{-a} + (\bar{m}\times\hat{m})^{-b} + \tilde{m}^{-c}\right)\\
		&\lesssim	 \inf_{m = \tilde{m}\times \hat{m} + \tilde{m}\times\hat{m} + \tilde{m}} \left( R^*_{\tilde{m},\hat{m}}(\F) _{L^2(0,T;H^{-1} (\Omega))} + R^*_{\bar{m}, \hat{m}}(\G)_{H^{1/2,1/4}(\Sigma)} + R^*_{\tilde{m}}(\M)_{L^2(\Omega)} \right) \\
		&\lesssim \inf_{m = |\X| + |\Y| + |\Z|} R^*(\U, \X, \Y, \Z)_{V} = R^*_m(\U)_{V},
	\end{align*}
	which completes the proof.
\end{proof}
\subsection{Concluding remarks on recovery rates}
This section derived the optimal recovery rates for the model classes \( \mathcal{F}, \mathcal{G}, \mathcal{M}, \mathcal{U} \), showing that several distinct classes yield the same rates. Consequently, smaller model classes are superfluous if they embed into larger ones with identical recovery performance. For instance, \( \mathcal{F} = \mathcal{U}(B^\theta_{p'q'}(0,T;B^s_{pq}(\Omega))) \) achieves the same optimal rate in \( L^2(0,T;H^{-1}(\Omega)) \) as in \( L^2(0,T;L^\delta(\Omega)) \) with \( \delta = 2d/(d+2) \), and attains the rate \( \tilde{m}^{-\alpha/d} \hat{m}^{-\beta} \) if it embeds into \( \mathcal{U}(B^\beta_{p'\infty} (0,T;B^\alpha_ {\gamma\infty}(\Omega))) \) and \( C(\Omega_T) \). The largest Besov-type model classes achieving optimal recovery rates for \(f, u_0, g\) are given as follows
\begin{align}\label{3.24}
	\begin{cases}
		\F&:= U(\B),\quad \B = B^\theta_{2,\infty}(0,T;B^s_{p,\infty}(\Omega)),\  p \geq \delta,\,  p' \geq 2,\, s > \frac{d}{p},\,  \theta > \frac{1}{p'},\\
		\mathcal{M} &:= \mathcal{U}(\B_0), \quad \B_0 = B^{\check{s}}_{p,\infty}(\Omega), \ \check{s}>d/\check p, 1 \leq \check p \leq \infty, \\
		\G &:= \text{Tr}(U(\overline{\B})),\quad \overline{\B} = B^{\bar\theta}_{2,\infty}(0,T;B^{\bar{s}}_{2,\infty}(\Omega)),\  1 \leq \bar{p}, \bar{p}' \leq 2,\, \bar{s} >d/\bar{p},
	\end{cases}
\end{align}
with \(\, \bar{\theta} > 1/\bar{p}', \, \bar{s}> \frac{d}{2},\, \bar{\theta}>\frac{1}{2}.\) Then the corresponding optimal recovery rate 
\begin{align}\label{3.25}
	\begin{cases}
		R^{*}_{m}(\F)_{L^2(0,T;H^{-1}(\Omega))} &\asymp \tilde{m}^{-\frac{s}{d}}\hat{m}^{\theta},\quad \tilde{m},\hat{m}\geq 1,\\
		R^*_{\tilde m}(\mathcal{M})_{L^2(\Omega)} &\asymp \tilde m^{-\check{s}/d}, \quad \tilde m \geq 1,\\
		R^{*}_{m}(\G)_{H^{1/2,1/4}(\Sigma)}&\asymp \bar{m}^{-\frac{\bar s-1}{d-1}}\hat{m}^{-\bar\theta}, \quad \bar{m}, \hat m\geq 1.
	\end{cases}
\end{align}

Based on the above discussion, from this point forward we fix the model classes by setting \( \F,\ \G\) and \( \M  \) as in \eqref{3.24}.

\textbf{Model classes for $u$:} Using model classes \(\F, \G, \M\) from \eqref{3.24} with \(\tilde{m} \asymp \bar{m}\), for $\tilde m, \hat{m} \geq 2$ the model class \(\U\) for \(u\) achieves the optimal recovery rate
\[
R^{*}_{m}(\U)_{V} \asymp \tilde m^{-\min\left\{\frac{s}{d}, \frac{\bar{s}-1}{d-1}, \frac{\check{s}}{d}\right\}}\hat{m}^{-\min\{\theta, \bar{\theta}\}}. \quad \tilde m, \hat m \geq 2.
\]
\section{Norms Discretization}\label{Sec_5}
This section introduces discrete norms based only on data site values. Replacing the norms in \( \mathcal{L}_T \) with these yields a numerically computable loss which is equivalent up to optimal recovery rates. 

\subsection{Discrete \( L^2(0,T;L^\tau(\Omega)) \) Norm}\label{sec_6.1}
We discretize the \( L^2(0,T;L^\tau(\Omega)) \) norm for \( 1 \leq \tau \leq \infty \), assuming \( f \in \F := U(B^{\theta}_{p'q'} (0,T;B^s_{pq}(\Omega))) \), with \( s > d/p \), \( \theta > 1/p' \). Consider the uniform grid \( G^{r,r'}_{k,k'} \subset \overline{\Omega} \), see \eqref{3.8}, with size \( \tilde{m} = [r2^k]^d[r'2^{k'}] \), where \( r > \max\{s,1\} \), \( r' > \max\{\theta,1\} \). For \( 1 \leq \tau < \infty \), define
\begin{align}
	\|f\|^{*}_{L^2(0,T;L^\tau(\Omega))}:=  \left[\frac{1}{\hat m}\sum_{j = 1}^{\hat m}\left[\frac{1}{\tilde m}\sum_{i = 1}^{\tilde m}|f(x_i,t_j)|^\tau \right] ^{\frac{2}{\tau}}\right]^{\frac{1}{2}},\quad \text{where} \, (x_i,t_j) \in G_{k,k'}^{r,r'}.
\end{align}
For \( \tau = \infty \), a standard modification applies. The lemma below shows the discrete \( L^2(0,T;L^\tau(\Omega)) \) norm closely approximates the true norm for model class \( \F \).

\begin{lemma}\label{lemma_6.1}
	Let \( f \in  B^{\theta}_{p'q'}(0,T;B^s_{pq}(\Omega)) \), with \( s > d/p \), \( \theta > 1/p' \), \( 1 \leq \tau \leq \infty \), and \( \Omega_T = (0,1)^d \times (0,T] \), the discrete and continuous norms satisfy
	\begin{align}
		\|f\|_{L^2(0,T;L^\tau(\Omega))} &\lesssim \|f\|^*_{L^2(0,T;L^\tau(\Omega))} + \|f\|_{\B}\tilde{m}^{-\alpha_\tau} \hat{m}^{-\beta_{\tau'}},\\
		\|f\|^*_{L^2(0,T;L^\tau(\Omega))} &\lesssim \|f\|_{L^2(0,T;L^\tau(\Omega))} + \|f\|_{\B}\tilde{m}^{-\alpha_\tau} \hat{m}^{-\beta_{\tau'}},
	\end{align}
	where \( \alpha_\tau := \frac{s}{d} - \left( \frac{1}{p} - \frac{1}{\tau} \right)_+ \) and \( \beta_\tau := \theta - \left( \frac{1}{p'} - \frac{1}{2} \right) \). 
\end{lemma}
\begin{proof}
	We prove the case \( 1 \leq \tau < \infty \), the argument for \( \tau = \infty \) is similar. Let \( S_{k,k'}^* \) denote the simplicial interpolant over \( \mathcal{T}_k(\Omega) \times \mathcal{I}_{k'} \), and apply \eqref{2.16}.
	\begin{align}
		\| S_{k,k'}^*(f) \|^2_{L^2(0,T;L^\tau(\Omega))} = \sum_{I^{'}\in \mathcal{I}_{k'}}^{}\left[\sum_{T \in \mathcal{T}_k} \| L_{E,I'}(f) \|_{L^2(0,T;L^{\tau}(T))}^\tau\right]^{\frac{2}{\tau}}, 
	\end{align}
	where \( L_{E,I'}(f) \in \mathcal{P}_{r,r'} \) interpolates \( \{ f(x_i,t_j) : (x_i,t_j) \in \bar E \times I' \} \). By norm equivalence 
	\begin{align}\label{6.3}
		\| L_{E,I'}(f) \|_{L^2(0,T;L^\tau(\Omega))} \asymp  \left[\frac{1}{\hat{m}}\sum_{j=1}^{\hat m}\left[\frac{1}{\tilde m} \sum_{x_i \in E} |f(x_i,t_j)|^\tau \right]^{2/\tau}\right]^{\frac{1}{2}}.
	\end{align}
	Summing over \( E \) and \( I' \), we obtain
	\begin{align}\label{6.4}
		\| S_{k,k'}^*(f) \|_{L^2(0,T;L^\tau(\Omega))} \asymp \hspace{-1mm}\left[\frac{1}{\hat{m}}\sum_{j=1}^{\hat m}\left( \frac{1}{m} \sum_{i=1}^{m} |f(x_i,t_j)|^\tau \right)^{\frac{2}{\tau}}\right]^{\frac{1}{2}} = \| f \|^*_{L^2(0,T;L^\tau(\Omega))}. 
	\end{align}
	For \( f \in \B \), the interpolant \( S_{k,k'}^*(f) \) achieves near-optimal recovery
	\begin{align}\label{6.5}
		\| f - S_{k,k'}^*(f) \|_{L^2(0,T;L^\tau(\Omega))} \leq C \| f \|_{\B} \tilde{m}^{-\alpha_\tau} \hat{m}^{-\beta_\tau}.
	\end{align}
	Hence
	\[ \| f \|_{L^2(0,T;L^\tau(\Omega))} - \| S_{k,k'}^*(f) \|_{L^2(0,T;L^\tau(\Omega))}| \leq C \| f \|_{\B}\tilde{m}^{-\alpha_\tau} \hat{m} ^{-\beta_\tau},\]
	and the result follows from \eqref{6.4} and \eqref{6.5}.
\end{proof}
\subsection{Discrete $H^{1/2,1/4}(\Omega)$ norm}\label{sec_6.2}
We define discrete \( H^{1/2,1/4}(\Sigma) \) norms that depend solely on function values at data sites, allowing direct computation. For any integers \( k, k' \geq 1 \), let \( \Y := \overline{G}_{k,k'}^{r,r'} = G_{k,k'}^{r,r'} \cap \Sigma \) denote the grid points on the lateral boundary \( \Sigma \), with total size \( \bar{m} \times \hat{m} = 2d[r2^k]^{d-1}[r'2^{k'}] \). For the model class \( \G = \operatorname{Tr}(U(\overline{\B})) \) as defined in \eqref{3.24}, the associated optimal recovery rate is given by
\begin{align}
	R^{*}_{\bar{m}}(\G)_{H^{1/2,1/4}(\Sigma)}  \asymp \bar{m}^{-\bar{\alpha}}\hat{m}^{-\beta_H},  \quad \bar{m}, \hat{m}\geq 1, \ \bar{\alpha}= \frac{\bar{s}-1}{d-1},\ \beta_H = \bar\theta.
\end{align}
This determines the smallest integers \( r > \max\{\bar{s},1\} \), \( r' > \max\{\bar{\theta},1\} \). For any continuous trace \( g = \operatorname{Tr}(v) \) with \( v \in \overline{\B} \), we define $\|g\|_{\operatorname{Tr}(\overline \B)} := \inf_{\operatorname{Tr}(v) = g} \|v\|_{\overline \B}.$ Then
\begin{align}
	|g|^2_{L^2(0,T;H^{1/2}(\partial\Omega))} &= \int_0^T \int_{\partial\Omega \times \partial\Omega} \frac{|g(x,t) - g(y,t)|^2}{|(x,t)-(y,t)|^d} \,dx\,dy\,dt,\label{6.9} \\
	|g|^2_{H^{1/4}(0,T;L^2(\partial\Omega))} &= \int_0^T \int_0^T \int_{\partial\Omega} \frac{|g(x,t) - g(x,s)|^2}{|(x,t)-(x,s)|^{3/2}} \,dx\,dt\,ds.\label{6.9.1}
\end{align}
Adding \( \|g\|_{L^2(0,T;L^2(\partial\Omega))} \) to both define the \( H^{1/2,1/4}(\Sigma) \) norm, known to be equivalent to the trace norm on \( \Omega_T \). Let \( \overline{S}_{k,k'} \) be as in \eqref{3.18}. From Theorem \eqref{thm_3.2} we have
\begin{align}\label{6.10}
	|\, \|g\|_{H^{1/2,1/4}(\Sigma)} - \|\overline{S}_{k,k'}(g)\|_{H^{1/2,1/4}(\Sigma)} | \leq C \|g\|_{\text{Tr}(\overline{\B})} \bar{m}^{-\bar{\alpha}} \hat{m}^{-\beta_H},
\end{align}
with constant \( C \) independent of \( \bar{m}, \hat{m}, g \). Hence, we now aim to construct a discrete \( H^{1/2,1/4}(\Sigma) \) norm for \( \overline{S}_{k,k'}(g) \), which depends only on values at the grid \( \overline{G}^{r,r'}_{k,k'} \). Define the dyadic partition \( \D_{k,k'}(\Sigma) := \Sigma \cap \D_{k,k'}(\bar{\Omega}_T) \), and let \( \T_{k,k'} := \T_{k,k'}(\Sigma) \) be the corresponding Kuhn-Tucker simplicial partition of \( \Sigma \). Let \( V^{r,r'}(\T_{k,k'}) \) be the space of continuous, piecewise polynomials of degree \( r, r' \) over \( \T_{k,k'} \), with \( \overline{S}_{k,k'} \) being the projector onto this space.
For \( g \in C(\Sigma) \), define discrete seminorms
\begin{align}\label{6.11}
	|g|^{*}_{L^2(0,T;H^{1/2}(\partial\Omega))} &:= \left[\frac{1}{\hat{m}\bar{m}^2}\sum_{l=1}^{\hat{m}}\sum_{i\neq j}^{}\frac{ |g(x_i,t_l)-g(x_j,t_l)|^2}{|(x_i,t_k)-(x_j,t_l)|^d}\right]^{1/2},\\
	|g|^{*}_{H^{1/4}(0,T;L^2(\partial\Omega))} &:= \left[\frac{1}{\hat{m}^2\bar{m}}\sum_{j\neq l}^{}\sum_{i}^{\bar{m}}\frac{ |g(x_i,t_l)-g(x_i,t_l)|^2}{|(x_i,t_j)-(x_i,t_l)|^{3/2}}\right]^{1/2},
\end{align}
and show that these discrete seminorms are equivalent to their continuous counterparts on \( V^{r,r'}(\T_{k,k'}) \), beginning with the following lemma.
\begin{lemma}\label{lemma_6.2}
	Let \( E \times I,\, E' \times I' \in \T_{k,k'} \) and \( S \in V^{r,r'}(\T_{k,k'}) \). Then,
	\begin{align}\label{6.12}
		\Bigg[\frac{1}{\bar{m}^2\hat{m}}\sum_{l=1}^{\hat{m}}\sum_{\substack{i \ne j}}^{}\frac{|S(x_i,t_l)- S(x_j,t_l)|^2} {|(x_i,t_l) -(x_j,t_l)| ^d} \Bigg]^{\frac{1}{2}} \asymp \left[\int_{0}^{T}\hspace{-0.2cm}\int_{E\times E'}^{}\frac {|S(x,t)- S(x',t)|^2} {|(x,t) - (x',t)| ^d}\, dx\,dx'dt\right]^{\frac{1}{2}},
	\end{align}
	for \((x_i,t_l) \in \bar E\times I,\, (x_j,t_l) \in \bar E'\times I\) and constants in \( \asymp \) depend on \( r, r' \), and \( d \).
\end{lemma}
\begin{proof}
	Fix \( d, r, r' \), and \( E \times I, E' \times I' \in \mathcal{T}_{k,k'}(\Sigma) \). Let \( R(S) \) and \( R'(S) \) be the discrete and integral expressions from \eqref{6.12}, defined on the finite-dimensional space \( X := X(E \times I, E' \times I') \subset V^{r,r'}(\mathcal{T}_{k,k'}) \). For \( k, k' \leq 4 \), both \( R \) and \( R' \) vanish only on constants over \( E \times I \cup E' \times I' \), and satisfy uniform norm equivalence on \( X \) with constants depending only on \( r, r', d \).
	\begin{align}\label{6.13}
		c R(S) \leq R'(S) \leq C R(S), \quad S \in X.
	\end{align}
	
	For \( k, k' > 4 \), define the best constants \( c(E, E', I, I') \) and \( C(E, E', I, I') \) for which \eqref{6.13} holds. We aim to show there exist uniform constants \( c^*, C^* > 0 \) (depending on \( r, r', d \)) such that $	c(E, E', I, I') \geq c^*, \ C(E, E', I, I') \leq C^*$.  We consider two cases:
	
	\textbf{Case 1:} \( \bar{E} \times \bar{I} \cap \bar{E}' \times \bar{I}' \neq \emptyset \). Then there exists a rigid transformation mapping a reference pair \( E_0 \times I_0, E_0' \times I_0' \in \mathcal{T}_{4,4} \) onto \( E \times I, E' \times I' \). Changing variables in \eqref{6.12} yields the same norm equivalence with constants \( c_1, C_1 \).
	
	\textbf{Case 2:} \( \bar{E} \times \bar{I} \cap \bar{E}' \times \bar{I}' = \emptyset \). From $L^2$ norm discretization, there exist constants \( c_2, C_2 > 0 \) (depending only on \( r, r', d \)) such that for all \( S \in V^{r,r'}(\mathcal{T}_{k,k'}(\Sigma) \),
	\begin{align*}
		&c_2 \left[ \int_{I}^{}\int_{E \times E'}\hspace{-1mm}|S(x,t) - S(x',t)|^2 \, dx \, dx' dt \right]^{\frac{1}{2}}\leq \\ \Bigg[ &\frac{1}{\bar{m}^2\hat{m}}\sum_{I\in \I_{k'}}^{} \sum_{\substack{x \neq x'}} |S(x,t)- S(x',t)|^2 \Bigg]^{\frac{1}{2}}\hspace{-1mm}\leq C_2 \left[  \int_{I}^{}\int_{E \times E'}\hspace{-1mm}|S(x,t) - S(x',t)|^2 \, dx \, dx' dt \right]^{\frac{1}{2}}.\notag
	\end{align*}
	In this setting, all pairwise distances \( |(x,t) - (x',t)| \) are uniformly comparable, so \eqref{6.12} holds with constants depending only on \( r, r', d \). This completes the proof.
\end{proof}
\begin{lemma}
	For all pair $E\times I, E'\times I'\in \T_{k,k'}$ and each $S\in V^{k,k'}(\T_{k,k'})$,
	\begin{align*}
		\Bigg[\frac{1}{\bar{m}\hat{m}^2}\sum_{\substack{j \ne l}}^{}\sum_{i = 1}^{\bar{m}}\frac{|S(x_i,t_j)- S(x_i,t_l)|^2} {|(x_i,t_j) -(x_i,t_l)| ^{3/2}} \Bigg]^{\frac{1}{2}} \asymp \left[\int_{I\times I'}^{}\int_{E}^{}\frac {|S(x,t)- S(x,t')|^2} {|(x,t) - (x,t')| ^{3/2}}\, dx\, dt\, dt'\right]^{\frac{1}{2}},
	\end{align*}
	for \((x_i,t_j) \in \bar E\times I,\, (x_i,t_l) \in \bar E\times I'\) and the constants in $\asymp$ depends on $r, r'$ and $d$.
\end{lemma}
\begin{proof}
	The proof proceeds similarly to Lemma \eqref{lemma_6.2}.
\end{proof}
We now show that the discrete and true semi-norms are equivalent for functions in $\G$.
\begin{theorem}
	Let \( g \in \G = \mathrm{Tr}(U(\overline{\B})) \), with \( \overline{\B} = B^{\bar\theta}_{\bar{p}',\infty}(0,T;B^{\bar{s}}_{\bar{p}, \infty}) \), where \( \bar{s} > d/2 \) and \( \bar{\theta} > 1/2 \). Then for all \( \bar{m}, \hat{m} \geq 1 \),
	\begin{align}\label{6.15}
		|g|_{H^{1/2,1/4}(\Omega)}&\lesssim |g|^*_{H^{1/2,1/4}(\Sigma)} + \|g\|_{\text{Tr}(\overline{\B})} \bar{m}^{-\bar{\alpha}} \hat{m}^ {-\beta_H}, \notag\\
		|g|^*_{H^{1/2,1/4}(\Omega)}&\lesssim |g|_{H^{1/2,1/4}(\Sigma)} + \|g\|_{\text{Tr}(\overline{\B})} \bar{m}^{-\bar{\alpha}} \hat{m}^ {-\beta_H},
	\end{align}
	with \( \bar{\alpha} = \frac{\bar{s}-1}{d-1} \), \( \beta_H = \bar{\theta} \), and constants depending only on \( r, r', d \).
\end{theorem}
\begin{proof}
	For any \( S \in V^{r,r'}(\T_{k,k'}) \), the continuous semi-norm is
	\begin{align}
		|S|_{H^{1/2,1/4}(\Sigma)} = &\left[\sum_{I\in \I_{k'}}^{}\sum_{E\times I,\,E'\times I\in \T_{k,k'}}^{}\int_{I}^{}\int_{E\times E'}^{}\frac{|S(x,t)-S(x',t)|^2}{|(x,t)-(x',t)|^d}\right]^{1/2} \notag\\&+  \left[\sum_{I,I'\in \I_{k'}}^{}\sum_{E\times I \in \T_{k,k'}}^{}\int_{I\times I'}^{}\int_{E}^{}\frac{|S(x,t)-S(x,t')|^2}{|(x,t)-(x,t')|^{3/2}}\right]^{1/2}. 
	\end{align}
	The discrete semi-norm satisfies
	\begin{align*}
		&|S|^*_{H^{1/2,1/4}(\Sigma)}\asymp\\
		&\Bigg[\sum_{\substack{(x,t)\ne (x',t)}} \hspace{-1mm} \frac{1}{\bar m^2\hat{m}}\frac{|S(x,t) - S(x',t)|^2 }{|(x,t) - (x',t)|^d}\Bigg]^{\frac{1}{2}}+  \Bigg[\sum_{\substack{(x,t)\ne (x,t')}} \hspace{-1mm}\frac{1}{\bar m\hat{m}^2}\frac{|S(x,t) - S(x,t')|^2 }{|(x,t) - (x,t')|^{3/2}}\Bigg]^{\frac{1}{2}},
	\end{align*}
	for each \(E\times I \in \mathcal{T}_{k,k'}\) and \((x,t) \in \overline G^{r,r'}_{k,k'}\). The equivalence above holds (up to constants depending only on \( r, r', d \)) due to shared data sites. By Lemma~\ref{lemma_6.2}, there exists \( c, C > 0 \) depending only on \( r, r', d \) such that
	\begin{align}\label{6.17}
		c \, |S|_{H^{1/2,1/4}(\Sigma)} \leq |S|^{*}_{H^{1/2,1/4}(\Sigma)} \leq C \, |S|_{H^{1/2,1/4}(\Sigma)}.
	\end{align}
	In particular, this holds for \( S = S_{k,k'}(g) \) with \( g \in \G \), and combining with \eqref{6.10} and the identity \( |g|^*_{H^{1/2,1/4}(\Sigma)} = |\overline{S}_{k,k'}(g)|^*_{H^{1/2,1/4}(\Sigma)} \) yields the result.
\end{proof}

To define a discrete norm on \( H^{1/2,1/4}(\Sigma) \), we introduce discrete \( L^2(0,T; L^2(\partial \Omega)) \) norm for functions in \( \mathcal{G} \), using the data sites \( \Y := G^{r,r'}_{k,k'} \)
\begin{align}\label{6.18}
	\|g\|^{*}_{L^2(0,T;L^2(\partial \Omega))} := \bigg[ \frac{1}{\bar{m} \hat{m}} \sum_{i=1}^{\bar{m}} \sum_{j=1}^{\hat{m}} |g(x_i, t_j)|^2 \bigg]^{1/2}, \quad g \in \mathcal{G}.
\end{align}
Following arguments from Section \ref{sec_6.1}, for \( \alpha = \frac{\bar{s} - 1/2}{d - 1} > \bar{\alpha} \), we obtain for all \( g \in \operatorname{Tr}(\overline{\B}) \),
\begin{align}\label{6.19}
	\|g\|_{L^2(0,T;L^2(\partial \Omega))} &\lesssim \|g\|^*_{L^2(0,T;L^2(\partial \Omega))} + \|g\|_{\operatorname{Tr}(\overline{\B})} \, \bar{m}^{-\alpha} \hat{m}^{-\beta_H}, \notag \\
	\|g\|^*_{L^2(0,T;L^2(\partial \Omega))} &\lesssim \|g\|_{L^2(0,T;L^2(\partial \Omega))} + \|g\|_{\operatorname{Tr}(\overline{\B})} \, \bar{m}^{-\alpha} \hat{m}^{-\beta_H}.
\end{align}
We now define the discrete \( H^{1/2,1/4}(\Sigma) \) norm as
\begin{align}\label{6.20}
	\|g\|^*_{H^{1/2,1/4}(\Sigma)} := 2\|g\|^*_{L^2(0,T;L^2(\partial \Omega))} + |g|^*_{H^{1/2,1/4}(\Sigma)}.
\end{align}
\begin{theorem}\label{thm_6.3}
	Let \( g \in \mathcal{G} = \operatorname{Tr}(U(\overline{\B})) \), with \( \overline{\B} = B^{\bar{\theta}}_{\bar{p}', \bar{q}'}(0,T; B^{\bar{s}}_{\bar{p}, \bar{q}}(\Omega)) \), where \( \bar{s} > d/2 \), \( \bar{\theta} > 1/2 \), and \( \bar{m}, \hat{m} \geq 1 \). Then,
	\begin{align}\label{6.21}
		\|g\|_{H^{1/2,1/4}(\Sigma)} &\lesssim \|g\|^*_{H^{1/2,1/4}(\Sigma)} + \|g\|_{\operatorname{Tr}(\overline{\B})} \, \bar{m}^{-\bar{\alpha}} \hat{m}^{-\beta_H}, \notag \\
		\|g\|^*_{H^{1/2,1/4}(\Sigma)} &\lesssim \|g\|_{H^{1/2,1/4}(\Sigma)} + \|g\|_{\operatorname{Tr}(\overline{\B})} \, \bar{m}^{-\bar{\alpha}} \hat{m}^{-\beta_H},
	\end{align}
	where the constants depend only on \( r, r' \), and \( d \).
\end{theorem}
\begin{proof}
	The result follows by combining proofs in \eqref{6.15} and \eqref{6.19}, with \( \alpha > \bar{\alpha} \).
\end{proof}

\subsection{A Discrete \( L^2(\Omega) \) Norm}
Let \( \mathcal{M} = U(B^{\check{s}}_{\check{p}\check{q}}(\Omega)) \) with \( \check{s} > \frac{d}{\check{p}} \), representing the model class for \( u_0 \). Consider the uniform grid \( G_{k,r} \subset \Omega \times \{t=0\} \) with \( \tilde{m} := |G_{k,r}| = [r2^k]^d \), where \( k \geq 0 \) and \( r > \max(s,1) \). For any continuous \( u_0 \), define the discrete norm
\[
\|u_0\|^{*}_{L^2(\Omega)} := \bigg( \frac{1}{\tilde{m}} \sum_{j=1}^{\tilde{m}} |u(x_j,t_0)|^2 \bigg)^{1/2}, \quad x_j \in G_{k,r}.
\]
\begin{lemma}\label{lemma_u00}
	Let \( \B_0 = B^{\check{s}}_{\check{p}\check{q}}(\Omega) \), \( \Omega = (0,1)^d \), and \( \check{s} > \frac{d}{\check{p}} \). Then for \( u_0 \in \B_0 \),
	\begin{align}
		\|u_0\|_{L^2(\Omega)} \lesssim \|u_0\|^{*}_{L^2(\Omega)} + \|u_0\|_{\B_0} \, \tilde m^{-\alpha},
		\
		\|u_0\|^{*}_{L^2(\Omega)} \lesssim \|u_0\|_{L^2(\Omega)} + \|f\|_{\B_0} \, \tilde m^{-\alpha},
	\end{align}
	with \( \alpha := \frac{s}{d} - \left( \frac{1}{p} - \frac{1}{2} \right) \), the optimal recovery rate of \( \mathcal{M} \) in \( L^2(\Omega) \). 
\end{lemma}
\begin{proof}
	Follows from \cite[Lemma~6.1]{BVPS} for \( L^\tau(\Omega) \) norms, applied here with \( \tau = 2 \).
\end{proof}

\subsection{Discrete loss under control}
We define a computable surrogate \( \L^{*} \) for the theoretical loss \( \L_T \), based on the available data \( (f, g, u_0) \). Fix parameters \( k, k', r, r' \), and let \( \X = G^{r,r'}_{k,k'} \subset \Omega_T \), \( \Y = \overline{G}^{r,r'}_{\bar k,k'} \subset \Sigma \), and \( \Z = G_{k,r} \subset \Omega \times \{0\} \) denote the data sites. Let \( \tilde{m}\times \hat{m} = \#(\X) \), \( \bar{m} \times \hat{m} = \#(\Y) \), and \( \tilde{m} = \#(\Z) \).
\begin{remark}
	Let \( \|\cdot\|^*_{L^2(0,T;L^\gamma(\Omega))} \) be the discrete norm from \eqref{6.3}, where \( \gamma = 2d/(d+2)\) for \(d\geq3\) and \(\gamma = 	(1 + [\log(m)]^{-1})\) for \(d=2\).
	Also \( \| \cdot \|^{*}_{H^{1/2,1/4}(\Sigma)} \) as in \eqref{6.11}, \eqref{6.18}, and \eqref{6.20}. For continuous \( v, \Delta v\) and \(v_t \) on \( \overline{\Omega}_T \), define the discrete loss
	\begin{align}\label{7.1}
		\L^{*}(v) :=
		\begin{cases}
			\|f + \Delta v - v_t\|^{*}_{L^2(0,T;L^{\gamma}(\Omega))} + \|g - v\|^{*}_{H^{1/2,1/4}(\Sigma)}, & d \geq 3, \\
			[1 + \log(\tilde m)] \|f + \Delta v - v_t\|^{*}_{L^2(0,T;L^{\gamma}(\Omega))} + \|g - v\|^{*}_{H^{\frac{1}{2},\frac{1}{4}}(\Sigma)},\hspace{-2mm} & d = 2.
		\end{cases}
	\end{align}
\end{remark}

The assumptions on $\{\F, \G, \M\}$ imply the solution \( u \) of \eqref{1.1} belongs to a model class \( \mathcal{U} \) with recovery rate $\max \left\{ \tilde m^{-s/d} \hat{m}^{-\theta}, \, \bar m^{-(\bar{s} - 1)/(d - 1)} \hat{m}^{-\bar{\theta}}, \, \tilde m^{-\check{s}/d} \right\}.$ We define $	\|v\|_{\mathcal{U}} := \max \left\{ \|\Delta v\|_{\B}, \, \|\operatorname{Tr}(v)\|_{\operatorname{Tr}(\B)}, \, \|u_0\|_{\B_0} \right\}$, for any function \( v \). The next theorem bounds \( \|u - v\|_{H^1(\Omega)} \) in terms of the discrete loss \( \L^{*}(v) \), under the assumption \( \Delta v \in \B,\, v \in \overline\B \), and \( u_0 \in \B_0 \).
\begin{theorem}
	Let \( u \) solve \eqref{1.1} with \( f \in \mathcal{F} = \mathcal{U}(\B) \), \( g \in \mathcal{G} = \operatorname{Tr}(\mathcal{U}(\B)) \), and \( u_0 \in \mathcal{M} = \mathcal{U}(\B_0) \), as in \eqref{3.24}. Given data at the sites \( G^{r,r'}_{k,k'} \), \( \bar{G}^{r,r'}_{\bar k,k'} \), and \( G_{k,r} \) with cardinalities \( \tilde{m} \times \hat{m} \), \( \bar{m} \times \hat{m} \), and \( \tilde{m} \) respectively, the discrete loss functional \( \L^* \) from \eqref{7.1} satisfies
	\begin{align}\label{7.3}
		\|u - v\|_V \lesssim \L^*(v) + \left(1 + \|v\|_{\mathcal{U}}\right)\mathcal{R}_{\mathcal{U}}(\tilde{m}, \bar{m}),
	\end{align}
	for any continuous \( v \in V \), where the constants are independent of \( u, v, \tilde{m}, \bar{m} \), and
	\begin{align}\label{7.4}
		\mathcal{R}_{\mathcal{U}}(\tilde{m}, \bar{m}) :=
		\begin{cases}
			\max \left\{ \tilde{m}^{-s/d} \hat{m}^{-\theta},\, \bar{m}^{-(\bar{s} - 1)/(d - 1)} \hat{m}^{-\bar{\theta}},\, \tilde{m}^{-\check{s}/d} \right\}, & d \geq 3, \\
			\max \left\{ (\log \tilde{m}) \tilde{m}^{-s/2} \hat{m}^{-\theta},\, \bar{m}^{-(\bar{s} - 1)} \hat{m}^{-\bar{\theta}},\, \tilde{m}^{-\check{s}/2} \right\}, & d = 2.
		\end{cases}
	\end{align}
\end{theorem}
\begin{proof}
	We prove the case \( d \geq 3 \). By \eqref{1.3},
	\begin{align*}
		\|u - v\|_V &\lesssim \|f + \Delta v - v_t\|_{L^2(0,T;H^{-1}(\Omega))} + \|g - \operatorname{Tr}(v)\|_{H^{1/2,1/4}(\Sigma)} + \|u_0 - v\|_{L^2(\Omega)} \\
		&\lesssim \L^*(v) + \|f + \Delta v - v_t\|_{\B} \tilde{m}^{-s/d} \hat{m}^{-\theta} + \|g - \operatorname{Tr}(v)\|_{\operatorname{Tr}(\B)} \bar{m}^{-(\bar{s} - 1)/(d - 1)} \hat{m}^{-\bar{\theta}} \\
		&\quad + \|u_0 - v\|_{\B_0} \tilde{m}^{-\check{s}/d}.
	\end{align*}
	The first inequality uses the weak formulation of the PDE and the second applies discrete to continuous norm estimates (see Lemma~\ref{sec_6.1}, Theorem~\ref{thm_6.3}, and Lemma~\ref{lemma_u00}). The result follows by bounding model norms assuming \( \|f\|_\B, \|g\|_{\operatorname{Tr}(\B)}, \|u_0\|_{\B_0} \leq 1 \). Similarly, the $d=2$ case will be directly followed by an integration in time from \cite[Theorem~7.2]{BVPS}.
	
\end{proof}

\section{Numerical results}\label{Sec_6}
In this section, we present numerical computation to verify the theoretical results developed in sections~\ref{sec_2} and \ref{sec_3}. These examples are designed to validate the consistent loss framework, the role of data-driven discretization, and improvements in the error behaviour of the solution as predicted by the analysis. The architecture and training setup used in our experiments follow the design given in \cite{BVPS}, which employs deep residual networks with $\text{ReLU}^3$ activation functions, a $\tanh$-activated input layer, and scaled weight initialization. We adopt the same configuration, including step size, momentum-based gradient descent, and rescaled velocities for computation for the space-time domain. For completeness, we refer to \cite{BVPS} for full implementation details of architecture for solving squared original loss function $\L_{sq}(v)$ and squared consistent loss function $\L^*_{sq}(v)$ corresponding to \eqref{1.1}.
\subsection{Experiments:}
In this experiment, we consider the problem \eqref{1.1} for two manufactured solutions $u_1$ and $u_2$. The source function \(f\), boundary data \(g\), and initial condition \(u_0\) are chosen so that:
\begin{align*}
	u_1(x,y,t) = xy(1-x)(1-y)e^{-t}, \qquad u_2(x,y,t) = \sin(\pi x)\cos(\pi y) + e^{-t},
\end{align*}
are the exact solutions. We solve this problem using a neural network with width \(W = 20\), depth \(L = 4\) (for $u_1$) and \(W = 100\) and \(L = 8\) (for $u_2$). 
The problem is more challenging for $u_2$ since \(f\), \(g\), and \(u_0\) are all nonzero.
\begin{table}[H]
	\centering
	\renewcommand{\arraystretch}{1.1}
	\begin{tabular}{|c|c c|c c|c|}
		\hline
		Mesh &\multicolumn{2}{|c|}{\textbf{Example 1}} &  \multicolumn{2}{c|}{\textbf{Example 2}} \\
		\cline{2-5}
		size & Rel err $(\%)$& Loss $(10^{-5})$ & Rel err $(\%)$ & Loss $(10^{-4})$ \\
		$N$ & PINNs / CPINNs & $\L_{sq}$ / $\L_{sq}^*$ & PINNs / CPINNs & $\L_{sq}$ / $\L_{sq}^*$ \\
		\hline
		$5$ & 21.77 / 4.75 & $4.0$ / $0.6$ &  3.01 / 1.06 & $1.71$ / $0.44$ \\
		$10$ & 14.52 / 3.03 & $1.5$ / $0.6$ &  2.72 / 0.44 & $1.84$ / $0.57$ \\
		$15$ & 14.96 / 3.09 & $1.5$ / $0.7$ & 2.36 / 0.41 & $1.43$ / $0.64$ \\
		$20$ & 16.06 / 3.01 & $1.7$ / $0.6$ &  2.26 / 0.38 & $1.33$ / $0.61$ \\
		$25$ & 14.99 / 2.54 & $1.4$ / $0.5$ &  2.16 / 0.36 & $1.20$ / $0.55$ \\
		$30$ & 14.18 / 2.23 & $1.2$ / $0.4$ & 2.27 / 0.35 & $1.31$ / $0.58$ \\
		\hline
	\end{tabular}
	\caption{Results demonstrating that the CPINNs loss function $\mathcal{L}^*_{sq}$ yields errors approximately 3 to 6 times smaller than those obtained by the original PINNs loss $\mathcal{L}_{sq}$.}
	\label{table_1}
\end{table}
Table~\ref{table_1} summarizes the results obtained for different mesh sizes of the form $(N\times N)$ corresponding to various choices of collocation points. 
This improvement is not due to better training dynamics, but rather to the superior error control provided by \(\mathcal{L}^{*}_{sq}\). While tuning the weights in \(\mathcal{L}_{sq}\) (e.g., varying \(\lambda\) in \eqref{1.5}) might improve its performance, our approach avoids such tuning by design, as all terms in \(\mathcal{L}^{*}_{sq}\) are naturally balanced.
\begin{figure}[H]
		\subfloat{
			{\includegraphics[width=0.3\textwidth]{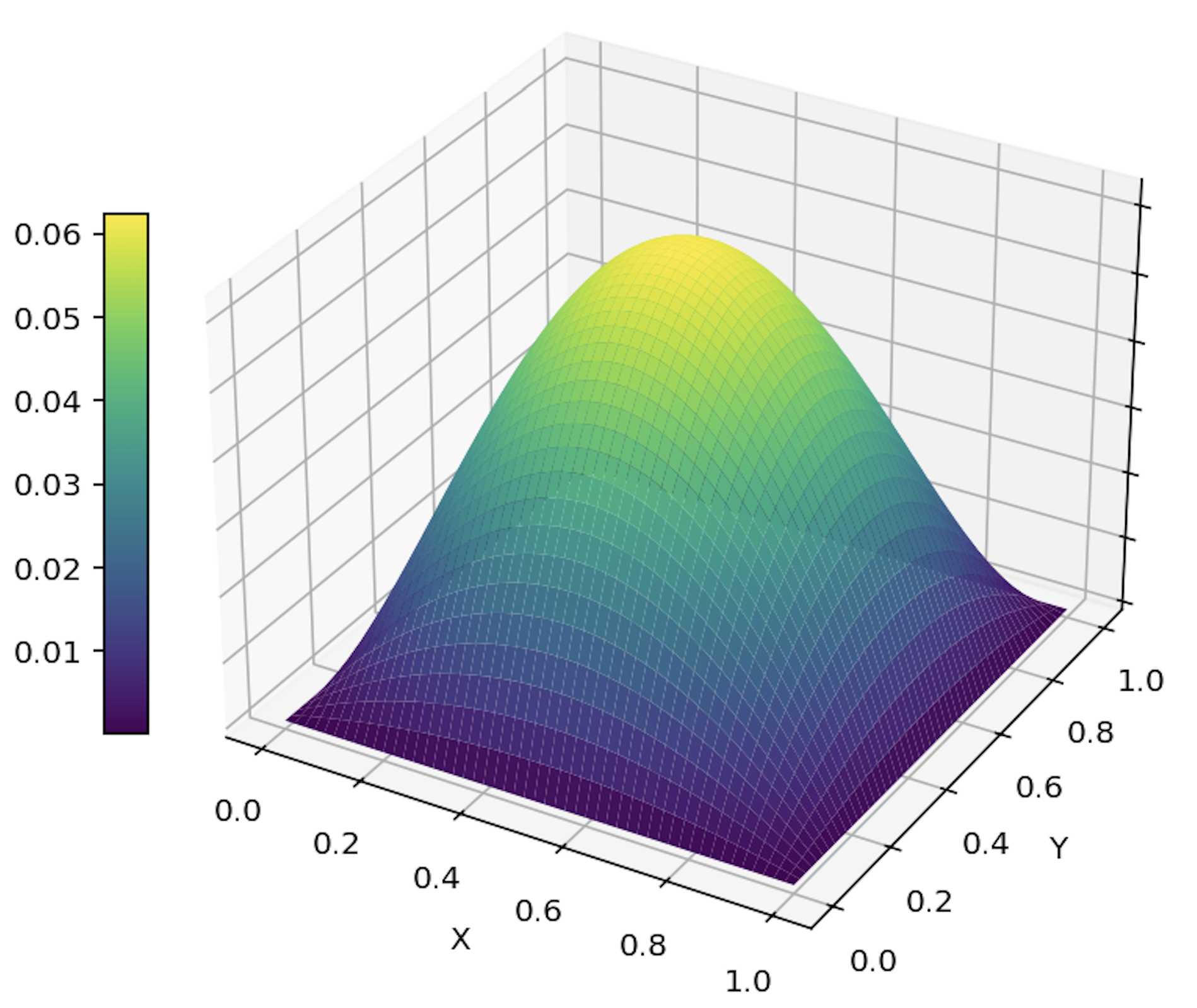}} 
			{\includegraphics[width=0.3\textwidth]{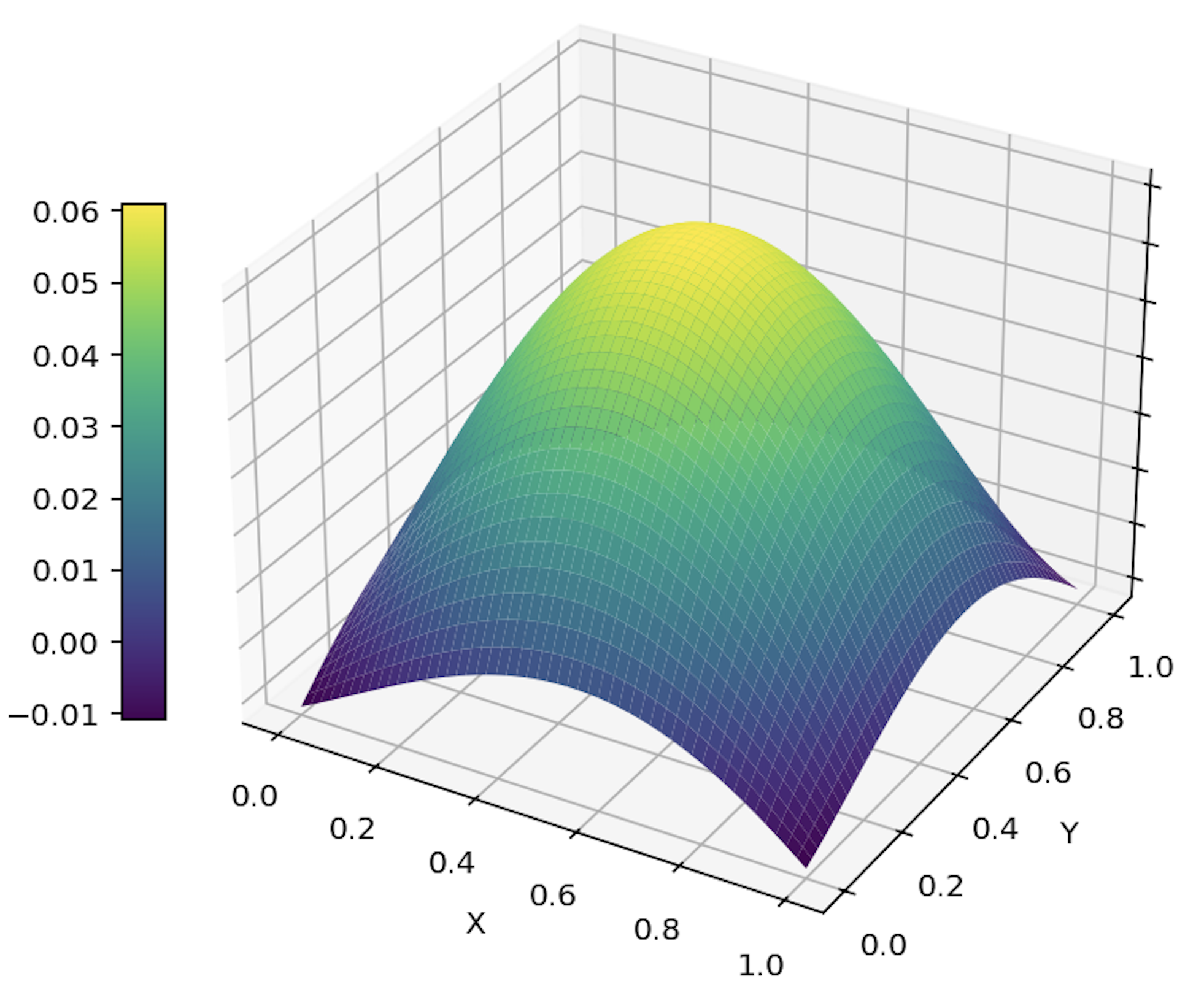}}
			{\includegraphics[width=0.3\textwidth]{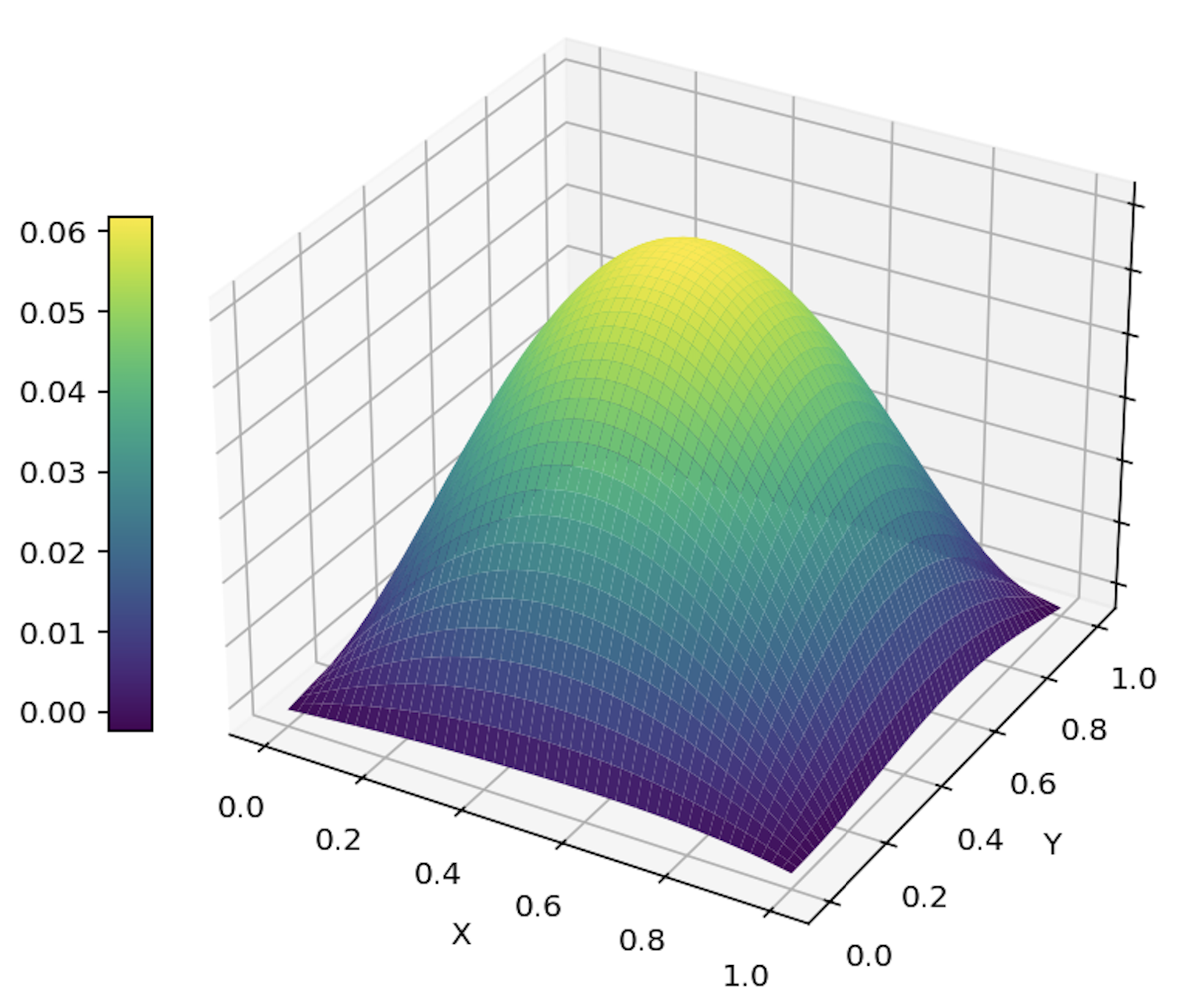}}
		}\\
		\subfloat{
			{\includegraphics[width=0.3\textwidth]{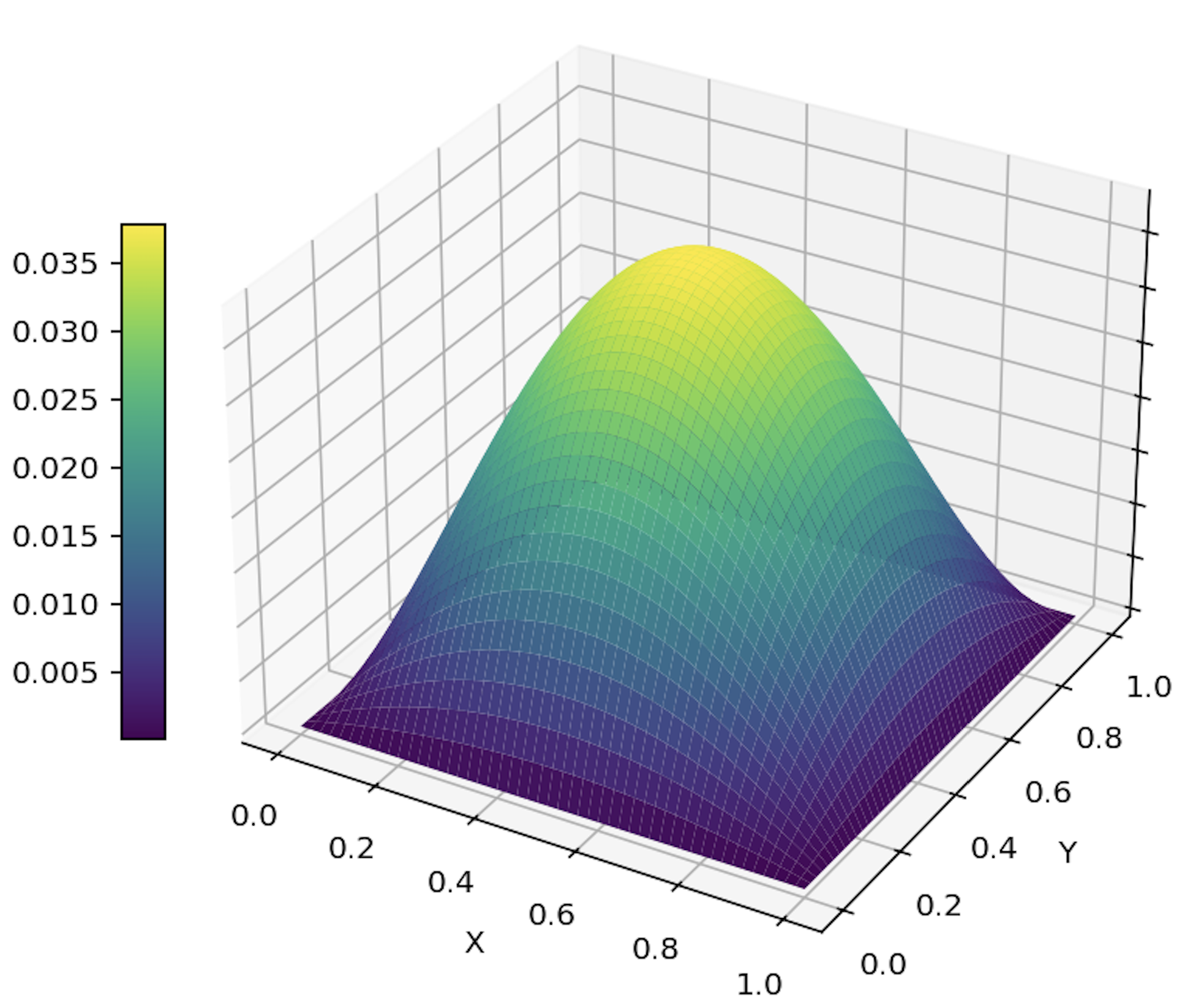}} 
			{\includegraphics[width=0.3\textwidth]{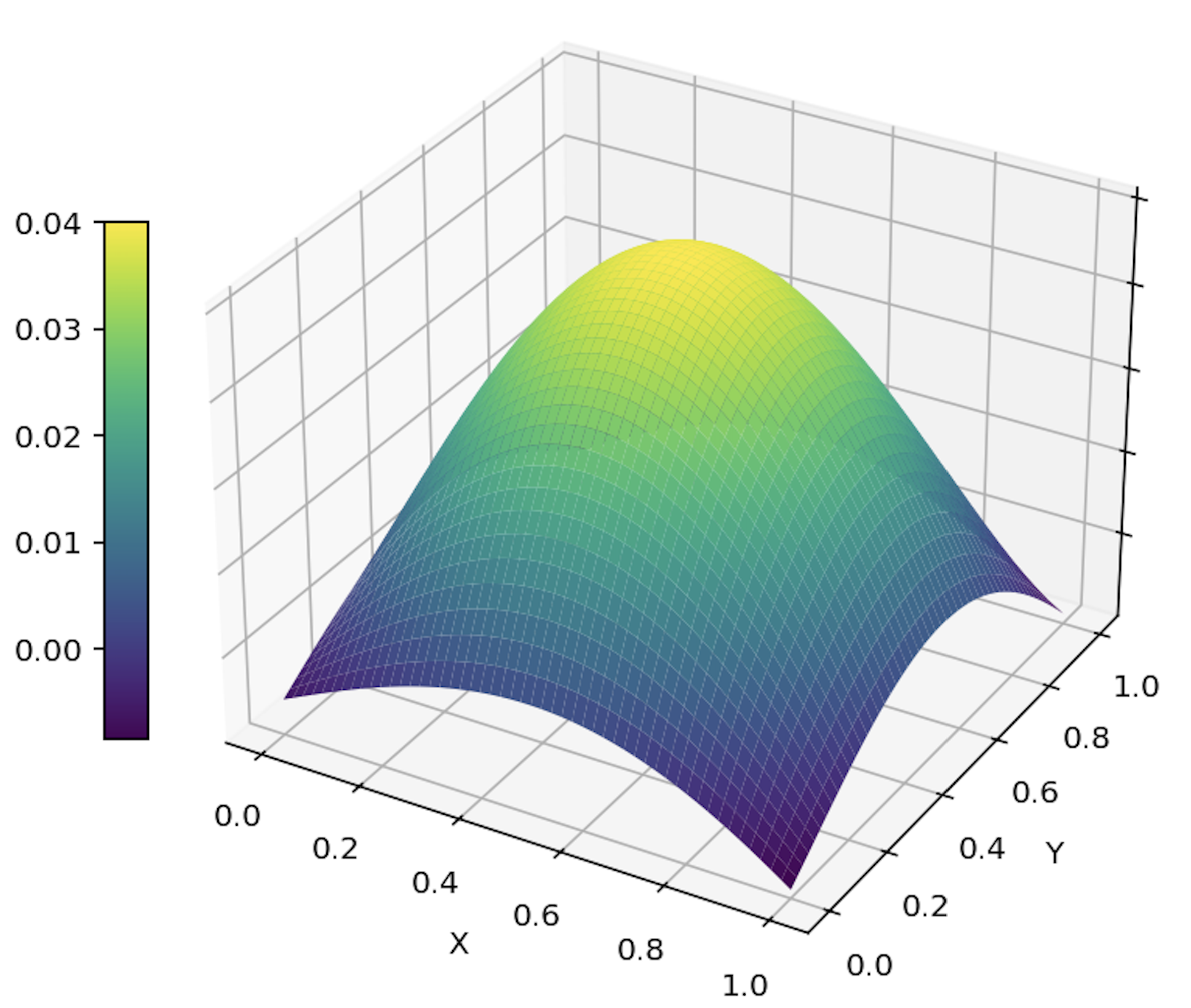}}
			{\includegraphics[width=0.3\textwidth]{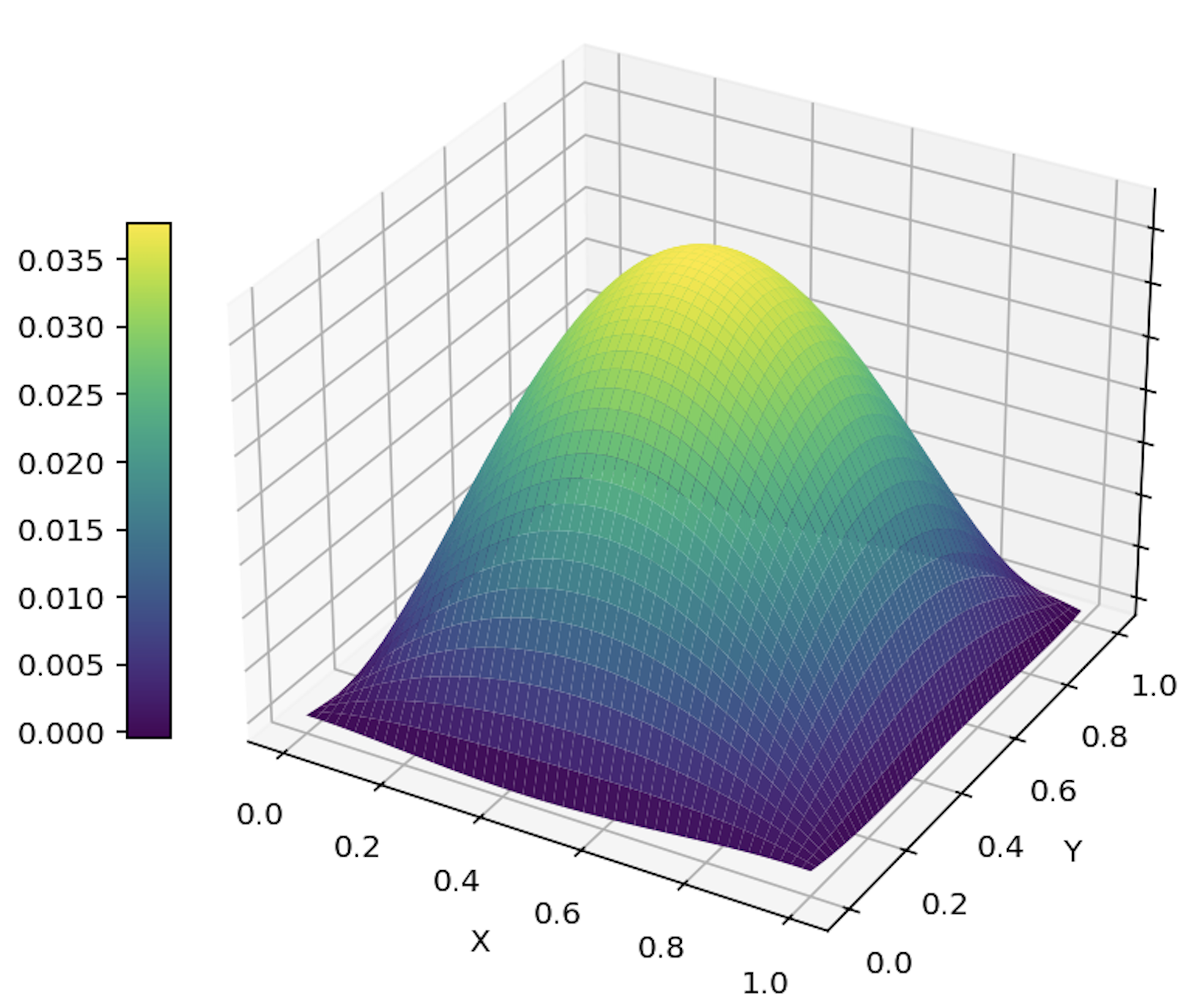}}
		}\\
		\subfloat{
			{\includegraphics[width=0.3\textwidth]{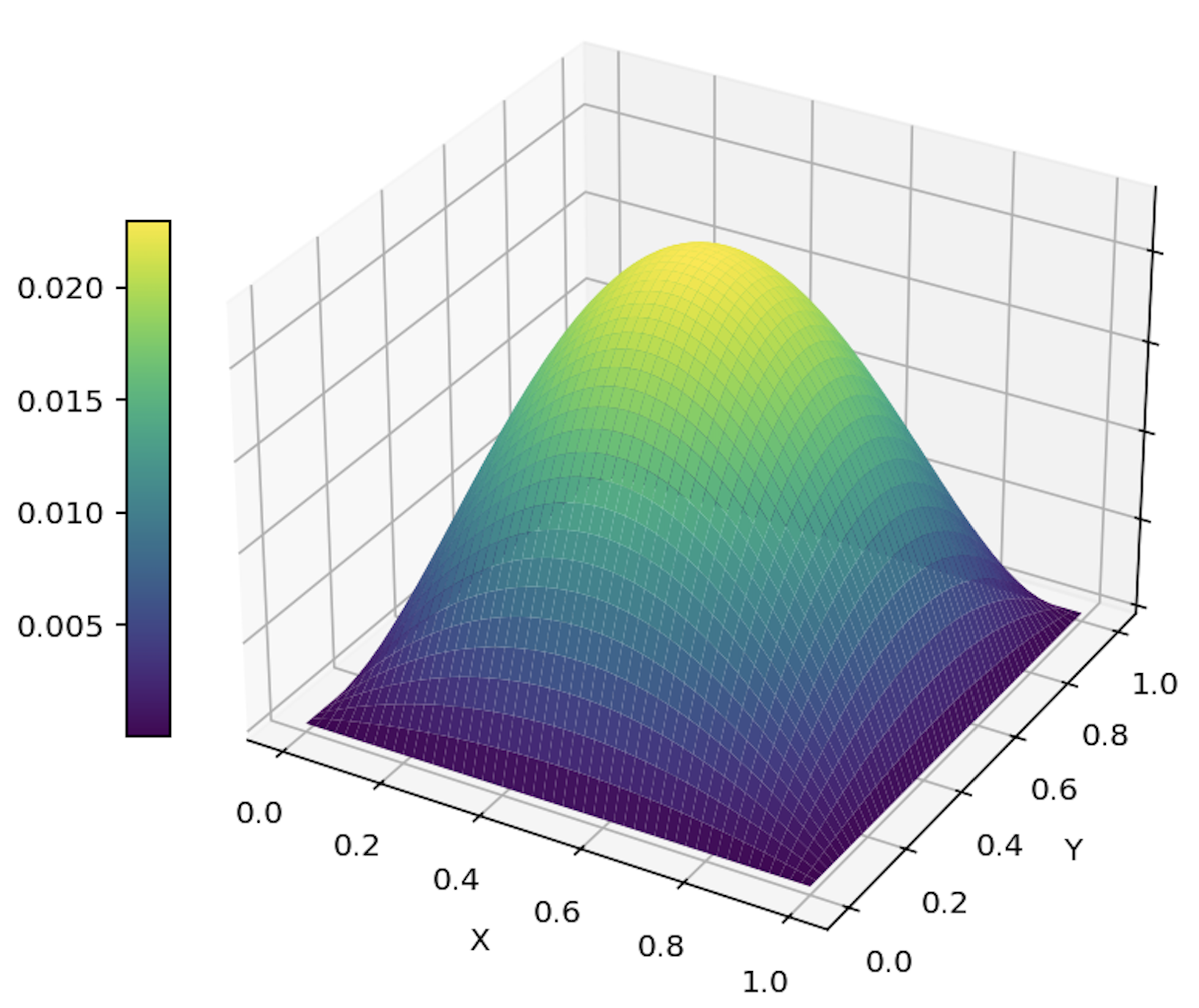}} 
			{\includegraphics[width=0.3\textwidth]{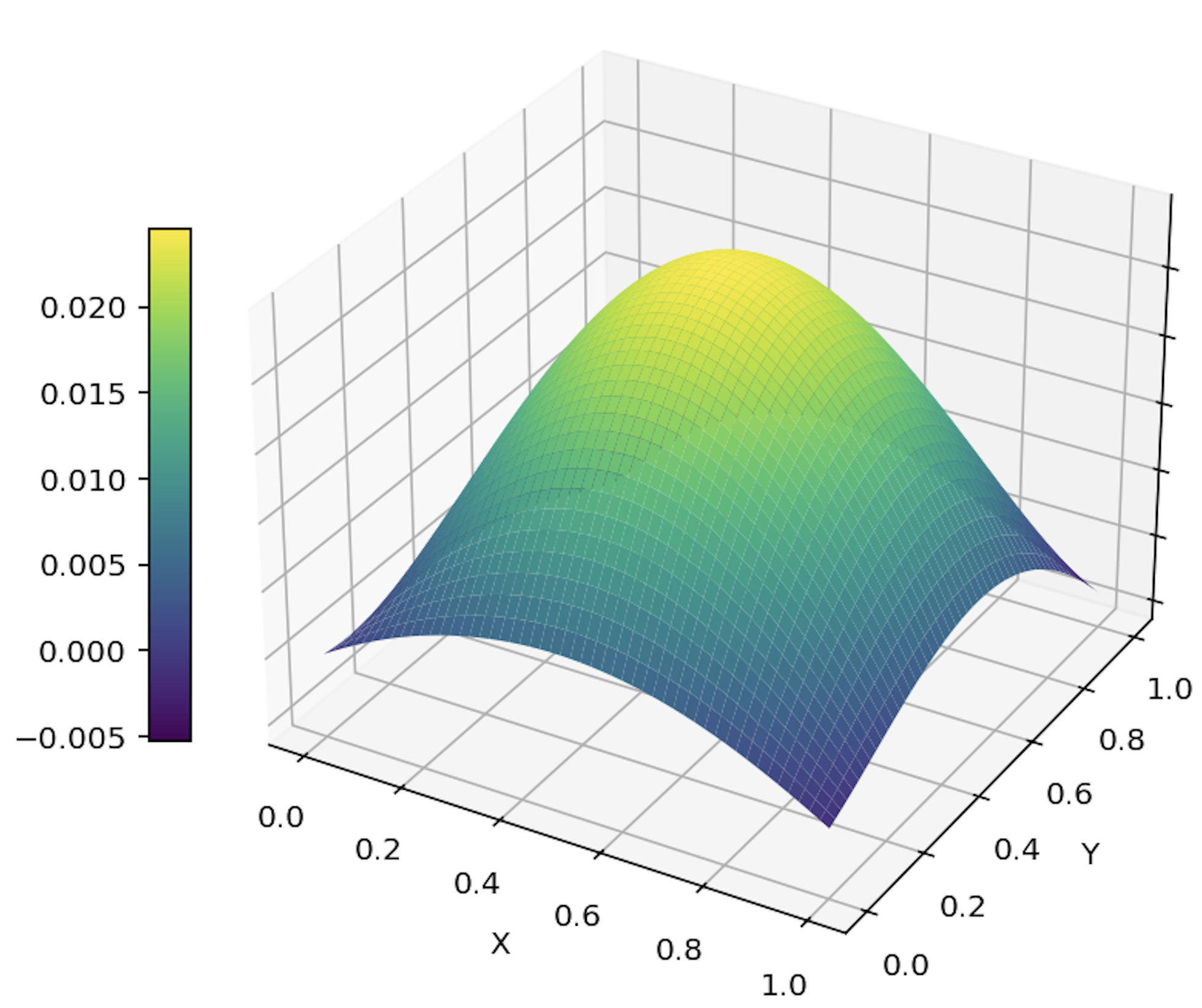}}
			{\includegraphics[width=0.3\textwidth]{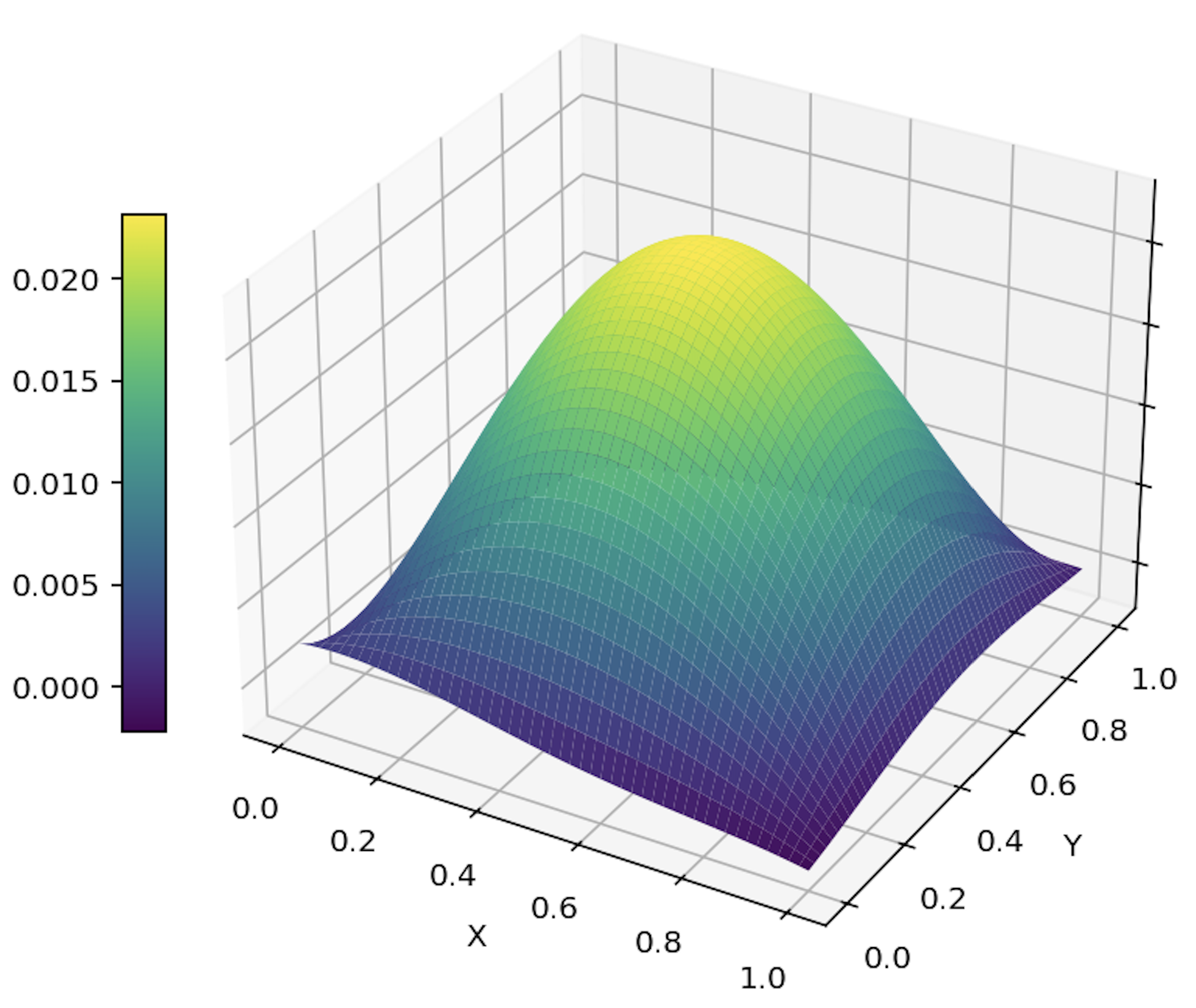}}
		}\\
		\caption{Plots of the exact solution $u_1$ (left), the learned solution using loss function $\L_{sq}$ (middle) and using loss function $\L_{sq}^{*}$ (right), for $15 \times 15$ collocation points at $t=0$ (top row), $t=0.5$ (middle row), and $t=1$ (top row).}
\end{figure}
\section{Conclusion}\label{Sec_7}
In this paper, we analyzed approximation properties of consistent PINNs for parabolic PDEs using non-linear approximation and optimal recovery. We established error estimates for piecewise polynomial approximations in mixed-norm Besov space \( B^{\theta}_{p'} (0,T; B^s_p(\Omega))\). These rates show the regularity of the target function and the resolution of the partition, which express how approximation accuracy scales with both spatial and temporal refinement. Develop near-best approximation property of local polynomial projectors, which enables robust error control. Our results not only justify the convergence of PINNs under mild assumptions but also bridge connections to classical approximation theory. Future work will extend these techniques to adaptive mesh refinements and nonlinear model classes.

\appendix

\section{Polynomial based local approximation}
Numerous key results exist for approximating functions in Besov spaces. In this work, we employ piecewise polynomial approximations, beginning with local polynomial approximation. For integers \( r, r' \geq 1 \), let \( \mathcal{P}_{r,r'} \) denote the space of polynomials of total degree less than \( r \) in the spatial variables and less than \( r' \) in the temporal variable.
\begin{align*}
	\mathcal{P}_{r,r'}:=\left\{\sum_{k'<r'}^{}\sum_{|k|_1<r}^{}a_{k,k'}x^kt^{k'},\ a_{k,k'}\in \mathbb{R}\right\},
\end{align*}
$\text{where } x^k := x_1^{k_1} \cdots x_d^{k_d},  k = (k_1, \dots, k_d),\, k_j \geq 0, \text{and }|k|_1 := \sum_{j=1}^{d} k_j.$
Note that $\omega_r(P,l\times l')_{L^{p'}(0,T;L^p(\Omega))}=0,\ l,l'\geq0,$ for all $P\in \mathcal{P}_{r,r'}$. If $I$ is any cube in $\mathbb{R}^d$, $I'$ be any interval in $[0,T]$ and $f\in L^{p'}(I';L^p(I)), 1<p\leq \infty$, we denote by
\begin{align}\label{12.1}
	\mathcal{E}_{r,r'}(f, I\times I')_{p.p'} := \inf_{P\in \mathcal{P}_{r,r'}}\|f-P\|_{L^{p'}(I';L^p(I))}.
\end{align} 
The quantity \( \mathcal{E}_{r,r'}(f, I \times I')_{p,p'} \) denotes the best approximation error of \( f \) on \( I \times I' \) in the mixed norm \( \| \cdot \|_{L^{p'}(I'; L^p(I))} \) by polynomials in \( \mathcal{P}_{r,r'} \). Whitney’s theorem states this error is equivalent to the mixed modulus of smoothness of \( f \), up to constants, for domains \( I \times I' \) with side lengths \( \ell_I, \ell_{I'} \).
\begin{align}\label{12.2}
	c \E_{r,r'} (f, I\times I')_{p,p'} \leq \omega_{r,r'}(f, \ell_I\times \ell_{I'})_{L^{p'}(I';L^p(I))} \leq C \E_{r,r'} (f, I\times I')_{p,p'},
\end{align}
for the constants $c,C$ rely solely on $r,d$ and $p,p'$, provided that $p,p'$ is near $0$.  The lower inequality in \eqref{12.2} is typically the only one to which Whitney's theorem applies. The upper inequality, however, arises trivially because for any polynomial $P \in P_{r,r'}$,
\begin{align*}
	\omega_{r,r'}(f, \ell_I\times \ell_{I'})_{L^{p'}(I';L^p(I))} \hspace{-0.7mm}=\hspace{-0.2mm} \omega_{r,r'}(f - \hspace{-0.4mm} P, \ell_I\times \ell_{I'})_{L^{p'}(I';L^p(I))}\hspace{-0.5mm}\leq \hspace{-0.3mm}C \|f - \hspace{-0.4mm}P\|_{L^{p'}(I';L^p(I))}.
\end{align*}
The following modified form of Whitney's theorem is helpful. For $0 <p<\infty$, $f\in \Lp,$ and $I\times I' \subset \mathbb{R}^d\times[0,T]$, we define
\begin{align*}\label{12.3}
	\tilde w_{r,r'}(f, b\times b')^{p'}_{\Lp} \hspace{-1mm}:= \frac{b^{-d} }{\delta b'}\int\limits_{h' \in [0,b']}\int\limits_{I_{r'h'}}\left[\int\limits_{h \in [-b, b ]^d} \int\limits_{I_{rh}} |\Delta^r_h\Delta^{r'}_{h'}f|^p \, dx \, dh \right]^{\frac{p'}{p}} dt  \,  dh',
\end{align*}
where, \( \delta b' \) denotes the length of the interval \( [0, b'] \), while \( I_{r'h'} := \{ t : [t, t + r' h'] \subset I' \} \) and \( I_{rh} := \{ x : [x, x + r h] \subset I \} \). The quantity \( \tilde{w}_{r,r'} \) is referred to as the averaged modulus of smoothness of \( f \), and it is equivalent to \( \omega_{r,r'} \).
\begin{align*}\label{12.4}
	c \tilde{w}_{r,r'} (f, b\times b')_{\Lp} \leq \omega_{r,r'}(f, b\times b')_{\Lp} \leq C \tilde{w}_{r,r'} (f, b\times b')_{\Lp}, 
\end{align*}
for every $0 < b,b' \leq 1$, where again the constants $c,C$ depend only on $r,r'$ and p and can be taken absolute when $r,r'$ is fixed and $0 <p_0 \leq p\leq \infty$ with $p_0$ fixed. Thus, Whitney’s theorem holds with $\omega_{r,r'}$ replaced by $\tilde w_{r,r'}$
\begin{align}\label{12.5}
	c \E_{r,r'} (f, I \times I')_{p,p'} \leq \tilde w_{r,r'} (f, \ell_I)_{\Lp} \leq C \E_{r,r'} (f, I\times I')_{p,p'}.
\end{align}
We use the averaged modulus \( \tilde{w}_{r,r'} \) as it satisfies subadditivity of \( \tilde{w}_{r,r'}^{p'} \), unlike \( \omega_{r,r'} \). For partitions \( \mathcal{I}, \mathcal{I'} \) of \( \Omega \) and \( [0,T] \), we will take advantage of this property in analysis.  
\begin{align}\label{12.6}
	\sum_{I' \in \mathcal{I'}}\sum_{I \in \mathcal{I}} \tilde w_{r,r'} (f, b\times b')^p_{\Lp} \leq \tilde w_{r,r'} (f, b\times b')^p_{L^p(0,T;L^p(\Omega))}, \quad b, b' > 0.
\end{align}
The subadditivity property also holds for partitions of \( \Omega_T \) into simplices \(\mathcal{T}\times \I'\). For \( I \times I' \subset \Omega_T \), a polynomial \( P_{I,I'} \in \mathcal{P}_{r,r'} \) is a near-best \( \Lp \) approximation to \( f \) with constant \( \lambda \geq 1 \) if
\begin{align}\label{12.7}
	\| f - P_I \|_{\Lp} \leq \lambda \E_{r,r'} (f, I\times I')_{r,r'}.
\end{align}
Lemma 3.2 in \cite{IBS} shows that if \( P_I \in \mathcal{P}_r \) is near-best in \( L^p(I) \) with constant \( \lambda \), then it remains near-best in \( L^{\bar{p}}(I) \) for all \( \bar{p} \geq p \). Extending this, if \( P_{I,I'} \in \mathcal{P}_{r,r'} \) is near-best in \( \Lp \) with constant \( \lambda \), it is also near-best in \( L^{\bar{p}'}(I';L^{\bar{p}}(I)) \).
\begin{align}\label{12.8}
	\| f - P_{I,I'} \|_{L^{\bar p'}(I';L^{\bar p}(I))} \leq C \lambda \E_{r,r'} (f, I\times I')_{p,p'},
\end{align}
with the constant $C$ depending only on $r,d,p,$ and $p'$. This constant does not depend on $I, I',\bar p$ or $\bar p'$. Any near-best approximation \( P_{I,I'} \) with constant \( \lambda \) is also near-best on any larger cuboid \( J \times J' \supset I \times I' \), satisfying
\begin{align}\label{12.9}
	\| f - P_{I,I'} \|_{L^{p'}(J';L^p(J))} \leq C \lambda \E_{r,r'} (f, J\times J')_{p,p'}.
\end{align}
where, \( C \) also depends on the ratio \( |J \times J'| / |I \times I'| \) (see Lemma 3.3 in \cite{IBS}). Although Lemmas 3.2 and 3.3 in \cite{IBS} are stated for polynomials of degree \( < r \), it holds for total degree \( < rr' \) ($r$ in space and $r'$ in time). In summary, near-best \( L^{p'}(I';L^p(I)) \) approximations remain near-best on larger cuboids \( J \times J' \supset I \times I' \) and for \( \bar{p} \geq p \). 

\subsection{Polynomial norms and inequalities}
We recall key norm equivalences (see (3.2) in \cite{BDS}, Lemmas 3.1 and 3.2 in \cite{RDVS}). For any cuboid \( I \times I' \subset \mathbb{R}^d \times [0,T] \) and \( g \in \Lp \), define the normalized norm
\begin{align}\label{12.10}
	\| g \|_{\Lp}^{\#} := |I|^{-1/p}|I'|^{-1/p'} \| g \|_{\Lp}. 
\end{align}
For any polynomial \( P \in \mathcal{P}_{r,r'} \), we have $	\| P \|_{L^{q'}(I';L^q(I))}^{\#} \asymp \| P \|_{\Lp}^{\#},$ with constants independent of \( I, I' \), provided \( 0<p_0\leq p,q,p',q' \leq \infty\). Moreover,
\begin{align}\label{12.11}
	|P|_{B_{p'q'}^{\theta} (I';B_{pq}^s (I))} \leq C \ell_I^{-s} \ell_{I'}^{-\theta}\| P \|_{\Lp},
\end{align} 
where \( C \) is independent of \( I, I' \) for \( P \in \mathcal{P}_{l,l'} \) and any \( s,\theta > 0 \) (see Corollary 5.2, \cite{IBS}).

\subsection{Besov spaces and piecewise polynomial approximation}
We recall that membership in \( B^{\theta}_p(0,T; B^s_p(\Omega)) \) implies approximation by piecewise polynomials at a certain rate. For \( k,k' \geq 0 \), let \( D_k \) and \( \mathcal{I}_{k'} \) be dyadic partitions of \( \Omega \) and \( [0,T] \), with mesh sizes \( 2^{-k} \) and \( 2^{-k'} \). Define \( \mathcal{S}_{k,k'}(r,r') \) as the space of piecewise polynomials of degree \( (r-1) \) in space and \( (r'-1) \) in time. On each cell \( I \times I' \), let \( P_{I,I'} \in \mathcal{P}_{r,r'} \) denote the best \( \Lp \) approximation to \( f \).
\begin{align}\label{12.13}
	S_{k,k'} := S_{k,k'}(f) := \sum_{I' \in \I_{k'}}\sum_{I \in \D_k} P_{I,I'} \chi_I \chi_{I'}\in S_{k,k'},
\end{align}
where \(\chi_I\) and \( \chi_{I'}\) are characteristic functions. 
The following lemma holds.

\begin{lemma}\label{L12.1}
	Let \(1 \leq p, p' \leq \infty\), \(s, \theta > 0\), and \(r, r' \geq 2\) with \(r > s\), \(r' > \theta\). If \(f \in \mathfrak{B} := B^\theta_{p'} (0,T; B^s_p(\Omega))\), then
	\begin{align}\label{12.14}
		\text{dist}(f,\S_{k,k'}(r,r'))_{L^{p'}(0,T;L^p(\Omega))}\leq C |f|_{\mathfrak{B}}2^{-(ks+k'\theta)}, \quad k,k'\geq0,
	\end{align}
	with constant \(C\) depending only on \(p, p', s\) and \(\theta\).
\end{lemma}

\begin{proof}
	Let \(S_{k,k'}\) be as in \eqref{12.13}. By using Whitney’s theorem and then for $p<\infty$ taking sum over $I \in \D_{k}$ we obtain
	\begin{align}\label{12.15}
		\|f(t) - P_{I,I'}(t)\|^p_{L^p(\Omega)} &\leq C^p\sum_{I \in D_k} \|\Delta^{r'}_{h'}\Delta^{r}_{h}f(t)\|^p_{L^p(I)} \leq C^P \|\Delta^{r'}_{h'}\Delta^{r}_{h}f(t)\|^p_{L^P(\Omega)}.
	\end{align}
	Now if $p'<\infty$, raise the power $p'/p$ both sides, and then integrate and sum it for $I'\in \I_{k'}$, we get summing over \(I \in D_k\) and integrating over \(I' \in \mathcal{I}_{k'}\), we obtain
	\begin{align}\label{12.16}
		\|f - &S_{k,k'}\|^{p'}_{L^{p'}(0,T;L^p(\Omega))} \leq C^{p'} \sum_{I' \in \I_{k'}}\int_{I'}^{}\|\Delta^{r'}_{h'}\Delta^{r}_{h}f(t)\|^{p'}_{L^p(\Omega)} \, dt \notag\\ &\leq C^{p'} \tilde w_{r,r'} (f, 2^{-k}\times 2^{-k'})^{p'}_{L^{p'}(0,T;L^p(\Omega)}\leq C^{p'} |f|^{p'}_{\mathfrak{B}} 2^{-(ks+k'\theta)p'},
	\end{align}
	where we used the subadditivity of the averaged modulus of smoothness and the Besov seminorm characterization. For case \(p = \infty\)  these inequalities follow directly from \eqref{12.15} and the fact that $\| \cdot \|_{\Lp} \leq \| \cdot \|_{L^\infty(0,T;L^\infty(\Omega))}$ for each \( I\times I' \subset \OMT \).
\end{proof}
\subsection{Local polynomial approximation in $L^{\tau'}(0,T;L^\tau(\Omega))$}
Here constants $C$ in this section may vary at each occurrence and dependent on $s, \theta, p$, and $p'$.
\begin{theorem}\label{thm_12.2}
	Let \( S_{k,k'} \) be as in \eqref{12.13}. If \( s > d/p,\ \theta > 1/p',\ 1\leq p \leq \tau \leq \infty,\ 1\leq p' \leq \tau' \leq \infty \), then for all \( f \in \mathfrak{B} :=\BV \),
	\begin{align}\label{12.17}
		\|f - S_{k,k'}\|_{L^{\tau'}(0,T;L^\tau(\Omega))} \leq C |f|_{\mathfrak{B}} 2^{-\left(k(s - \frac{d}{p} + \frac{d}{\tau})+ k'(\theta- \frac{1}{p'}+ \frac{1}{\tau'}) \right)}.			
	\end{align}
\end{theorem}

\begin{proof}
	Let us fix any \( f \in \BV \) and consider the corresponding \( S_{k,k'} = S_{k,k'}(f) \), see \eqref{12.13}.  
	It was proven in Lemma \eqref{L12.1} that  
	\[\|f - S_{k,k'}\|_{L^{p'}(0,T;L^p(\Omega))} \leq C |f|_{\mathfrak{B}} 2^{-(k s+k'\theta)}, \quad k,k' \geq 0.\]
	Let $R_{0,0} := S_{0,0}$ and for each $(x,t) \in \OMT$, we define
	\begin{align}\label{12.18}
		R_{k,k'} (x,t) &:= S_{k,k'} (x,t) - S_{k-1,k'}(x,t) - S_{k,k'-1} (x,t) - S_{k-1,k'-1} (x,t), \quad k,k' \geq 1,\notag\\
		R_{k,0}(x,t) &\hspace{-1mm}:= S_{k,0}(x,t) - S_{k-1,0}(x,t),\,
		R_{0,k'}(x,t) \hspace{-1mm}:= S_{0,k'}(x,t) - S_{0,k'-1}(x,t)\ k, k'\geq 1.\notag
	\end{align}
	The functions $R_{k,k'}$ are defined for all $(x,t) \in \bar\Omega\times [0,T]$ and are in $S_{k,k'} , k,k'\geq 0$,
	\begin{align}\label{12.19}
		\|R_{k,k'}\|_{L^{p'}(0,T;L^p(\Omega))} &\leq C \big[ \|f - S_{k,k'}\|_{L^{p'}(0,T;L^p(\Omega))} + \|f - S_{k-1,k'}\|_{L^{p'}(0,T;L^p(\Omega))} \notag\\&\hspace{0.5cm}+  \|f - S_{k,k'-1}\|_{L^{p'}(0,T;L^p(\Omega))} +  \|f - S_{k-1,k'-1}\|_{L^{p'}(0,T;L^p(\Omega))} \big] \notag\\ 
		&\leq C |f|_{\mathfrak{B}} 2^{-(k s+k'\theta)}, \quad k,k' \geq 1,
	\end{align}
	and of course $\|R_{0,0}\|_{L^{p'}(0,T;L^p(\Omega))} \leq C \|f\|_{L^{p'}(0,T;L^p(\Omega))}$. It follows that
	\begin{align}\label{12.20}
		f = \sum_{k'=0}^\infty\sum_{k=0}^{\infty} R_{k,k'},
	\end{align}
	with the series converging in $L^{p'}(0,T;L^p(\Omega))$. Consider the following cases for $\tau, \tau'$.\\
	\textbf{Case 1: $\tau, \tau' = \infty$,} For every dyadic cuboid $I\times I' \in \D_k\times \I_{k'}$ , we have
	\[ R_{k,k'} (x) = Q_{I,I'} (x) := P_{I,I'} (x,t) - P_{\bar I,\bar I'} (x,t), \quad (x,t) \in I\times I',\]
	where $(\bar I , \bar I') \in \D_{k-1}\times  \I_{k-1}$. From Whitney’s theorem and \eqref{12.11}, we have 
	\begin{align}\label{12.21}
		\|Q_{I,I'}\|_{C(I\times I')} &\leq C |I|^{1/p} |I'|^{-1/p} \|Q_{I,I'}\|_{\Lp} \notag\\&\leq C 2^{kd/p+k'/p'} \left[ \|f - P_{I,I'}\|_{\Lp} + \|f - P_{\bar I,\bar I'}\|_{\Lp} \right]\notag\\
		&\leq C 2^{(kd/p+k'/p')} \omega_{r,r'}(f, 2^{-k+1}\times 2^{-k'+1})_{\Lp}.
	\end{align}
	Thus for every $(x,t) \in \bar \Omega\times [0,T]$, we have
	\begin{align}\label{12.22}
		|R_{k,k'}(x,t)| &\leq C 2^{(kd/p+k'/p')} \omega_{r,r'}(f, 2^{-k+1}\times 2^{-k'+1})_{L^{p'}(0,T;L^p(\Omega))}\notag\\ &\leq C |f|_{\mathfrak{B}} 2^{-k(s - d/p)-k'(\theta-1/p')}, \quad k,k' \geq 1.
	\end{align}
	The series \eqref{12.20} converges in \( L^\infty(0,T;L^\infty(\Omega)) \) and also pointwise to a limit function \( f \), and for each \( (x,t) \in \OMT,\  k,k' \geq 0 \), we have  
	\begin{align}\label{12.23}
		|\tilde f - S_{k,k'} (x,t)| \leq \hspace{-1mm}\sum_{j'>k'}^{}\sum_{j>k}^{}|R_{j,j'} (x,t)| \leq C |f|_{\BV} 2^{-k(s - \frac{d}{p})-k'(\theta-\frac{1}{p'})},			
	\end{align}
	which proves the theorem for \( \tau = \tau' = \infty \).\\
	\textbf{Case 2: $\tau, \tau' <\infty$,} As in \eqref{12.21} and polynomial norms comparison of \eqref{12.11} gives
	\begin{align*}
		\|R_{k,k'}&\|^{\tau'}_{L^\tau (0,T;L^\tau(\Omega))}= \int_{0}^{T}\left[\int_{\Omega}^{}| R_{k,k'} |^{\tau}\ dx\right]^{\frac{\tau'}{\tau}}dt = \sum_{I'\in \I_{k'}}^{} \int_{I'}^{}\left[\sum_{I \in \D_k}  \int_{I}^{}| Q_{I,I'} |^{\tau}dx\right]^{\frac{\tau'}{\tau}}dt\\ 
		&\leq2^{d\tau'/p} C^{\tau'} 2^{kd\tau'(\frac{1}{p} - \frac{1}{\tau})} \sum_{I'\in \I_{k'}}^{} \int_{I'}^{}\|\Delta_{h}^{r}\Delta_{h'}^{r'} f(t)\|^{\tau'}_{L^p(\Omega)}\\
		&\leq2^{\tau'(\frac{d}{p}+\frac{1}{p'})} C^{\tau'} 2^{\tau'\left(kd(\frac{1}{p} - \frac{1}{\tau})+k'(\frac{1}{p'} - \frac{1}{\tau'})\right)}  \tilde w_{r,r'} (f, 2^{-k+1}\times 2^{-k'+1})^{\tau}_{ L^{p'}(0,T;L^p(\Omega)}.
	\end{align*}
	In other words, for any $\tau \geq p$ and $\tau'\geq p'$, we have
	\begin{align*}\label{12.24}
		\|R_{k,k'}\|_{L^{\tau'} (0,T;L^\tau(\Omega))} \leq C\, 2^{\left(-(ks+k'\theta)+kd(\frac{1}{p} - \frac{1}{\tau})+k'(\frac{1}{p'} - \frac{1}{\tau'})\right)}  |f|_{\mathfrak{B}}, \, k,k' \geq 1.
	\end{align*}
	Since \( f - S_{k,k'} = \sum_{j' > k'}\sum_{j > k} R_{j,j'} \),
	\begin{align*}
		\|f - S_{k,k'}\|_{L^\tau (0,T;L^\tau(\Omega))}& \leq  \sum_{j' > k'}\sum_{j > k} \|R_{j,j'}\|_{L^\tau (0,T;L^\tau(\Omega))}\\
		&\leq C |f|_{\mathfrak{B}} 2^{-\left(k(s - \frac{d}{p} + \frac{d}{\tau}) + k'(\theta -\frac{1} {p'} + \frac{1} {\tau'})\right)},
	\end{align*}
	which completes the proof in this case.
\end{proof}
\begin{remark}
	For  \( C \) independent of \( I, I' \), Theorem~\eqref{thm_12.2} also applies to general \( S_{k,k'} \in \S_{k,k'} \) with
	\begin{align}\label{12.25}
		\|f - P_{I,I'}\|_{L^{\tau'} (I';L^\tau(I))} \leq C 2^{k d(\frac{1}{p} - \frac{1}{\tau}) + k'(\frac{1}{p'} - \frac{1}{\tau'})} \tilde w_{r,r'} (f, 2^{-(k+k')})_{\Lp}.
	\end{align}
\end{remark}
\section*{Data availability}
The datasets produced during this study will be available from the corresponding author upon reasonable request.

\bibliographystyle{plain}
\bibliography{reference_2n}
\end{document}